\documentclass[11pt,reqno,a4paper]{amsart}
\usepackage{amsmath,amsfonts,amssymb,amsthm,amsxtra,latexsym,amscd,enumerate,verbatim}
\usepackage[margin=2.0cm]{geometry}
\usepackage[abbrev,lite,nobysame]{amsrefs}
\usepackage{graphics,graphicx}
\usepackage{subcaption}
\usepackage{tabularx}
\usepackage[usenames,dvipsnames]{color}
\usepackage{times}
\usepackage{bbm}
\usepackage[colorlinks=true, pdfstartview=FitV, linkcolor=blue, citecolor=blue, urlcolor=blue]{hyperref}
\usepackage{tikz}
\usepackage{enumerate,enumitem}
\usepackage{mathtools}
\makeatletter
\usepackage[utf8]{inputenc}
\usepackage{amsmath,amsfonts,amssymb,amsxtra,latexsym,amscd,enumerate,amsthm,verbatim}
\usepackage[abbrev,lite,nobysame]{amsrefs}
\usepackage{float}
\usepackage{tikz}
\usepackage{pgfplots}
\usepackage{caption}
\usepackage[colorlinks=true]{hyperref}
\hypersetup{urlcolor=blue, citecolor=blue}
\usepackage{hyperref}
\numberwithin{equation}{section}
\def\dd{{\mathrm{d}}}

\newcommand{\R}{\mathbb{R}}

\newcommand{\C}{\mathbb{C}}
\newcommand{\Z}{\mathbb{Z}}
\DeclareMathOperator{\diam}{diam}
\newcommand{\defeq}{\mathrel{\mathop:}=}

\newcommand{\inn}{{\quad\hbox{in } }}
\newcommand{\ve}{\varepsilon}
\newcommand{\be}{\begin{equation}}
\newcommand{\ee}{\end{equation}}

\newtheorem{lemma}{Lemma}[section]
\newtheorem{prop}{Proposition}[section]
\newtheorem{theorem}{Theorem}

\newtheorem{remark}{Remark}[section]

\pgfplotsset{compat=1.18} 

\title[Helical Vortex Filaments with compactly supported cross-sectional vorticity]{Helical Vortex Filaments with compactly supported cross-sectional vorticity for the Incompressible Euler Equations in $\R^3$}
\author[A.~Averkiou]{Averkios Averkiou}
\address{\noindent A.~Averkiou: Department of Mathematical Sciences University of Bath, Bath BA2 7AY,
United Kingdom.}
\email{aa4119@bath.ac.uk}

\author[M.~Musso]{Monica Musso}
\address{\noindent M.~Musso: Department of Mathematical Sciences University of Bath, Bath BA2 7AY,
United Kingdom.}
\email{m.musso@bath.ac.uk}

\begin{document}
\begin{abstract}
 We revisit the vortex filament conjecture for three-dimensional inviscid and incompressible Euler flows with helical symmetry and no swirl. Using gluing arguments, we provide the first construction of a smooth helical vortex filament in the whole space $\R^3$ whose cross-sectional vorticity is compactly supported in $\R^2$ for all times. The construction extends to a multi-vortex solution comprising several helical filaments arranged along a regular polygon. Our approach yields fine asymptotics for the vorticity cores, thus improving related variational results for smooth solutions in bounded helical domains and infinite pipes, as well as non-smooth vortex patches in the whole space.
\end{abstract}
\maketitle
\section{Introduction}
We consider the incompressible Euler equations in $\R^3$, which govern the evolution of an incompressible and inviscid fluid with constant density. Given a smooth initial velocity field $u_0(x)$, the Cauchy problem for the Euler system in vorticity formulation reads 
\begin{equation}
\label{euler}
\begin{aligned}
\vec \omega_t  +
(\mathbf{u}\cdot \nabla){\vec \omega}
&=( \vec \omega \cdot \nabla)\mathbf {u}  &&  \inn \, \R^3\times (0,T), \\
\quad \mathbf{u}  = \nabla \times \vec {\psi},\ &
-\Delta \vec \psi =  \vec \omega  && \inn \, \R^3\times (0,T),\\
{\vec \omega}(\cdot,0)  &=  \nabla \times u_0 && \inn \, \R^3,
\end{aligned}
\end{equation}
where $\textbf{u} : \R^3\times [0,T) \to \R^3$ represents the velocity field,\, $\vec{\omega}=\nabla\times{\mathbf{u}}: \R^3\times [0,T) \to \R^3$ denotes the vorticity field, and $\vec{\psi}: \R^3 \times [0,T) \to \R^3$ is the corresponding stream function of the fluid.

We study smooth solutions of \eqref{euler} with highly concentrated vorticity, particularly those known as \emph{vortex filaments}, corresponding to solutions whose vorticity is sharply concentrated in a thin tube around a smooth evolving curve in $\R^3$. The rigorous analysis of such flows is challenging and dates back to the works of Helmholtz \cite{helmholtz1858integrale} and Kelvin. 

In their classical works \cites{da1906sul,levi1908sull}, Da Rios and Levi-Civita used potential theory to formally show that if the vorticity concentrates in a tube of radius $\ve>0$ around a curve $\mathcal{G}(t)$, then its motion follows a binormal curvature flow with speed of order $|\log\ve|$. Specifically, for an arclength parametrization $\gamma(s,t)$, they derived
\[\partial_{t}\gamma=2\bar{c}|\log\ve|\left(\gamma_{s}\times\gamma_{ss}\right)\ \quad \text{as}\quad \ve\to 0,\]
where $\bar c$ is a circulation constant. Rescaling time via $t = \frac{\tau}{|\log \ve|}$ yields the binormal flow
\begin{equation}\label{binormal}
\partial_{\tau}\gamma=2\bar{c}\kappa\mathbf{B}_{{\mathcal{G}(\tau)}},
\end{equation}
where $\kappa$ is the curvature and $\mathbf{B}_{\mathcal{G}(\tau)}$ the binormal unit vector. This result identifies \eqref{binormal} as the effective dynamics of a vortex filament, in analogy with the Helmholtz-Kirchhoff system for $2$D point vortices \cites{MR1867882, marchioro2012mathematical, ricca1996contributions}. Under an a priori concentration assumption, Jerrard and Seis \cite{MR3609248} provided the first rigorous justification of \eqref{binormal}.

In this context, a central open problem is the \emph{vortex filament conjecture}, which posits that solutions of \eqref{euler} initially concentrated near a curve evolving by \eqref{binormal} remain concentrated for some time. Although still open in general, significant progress has been achieved for certain symmetric curves by using explicit solutions of \eqref{binormal}, such as translating circles and helices.

For the case of circular filaments, Fraenkel \cite{fraenkel1970steady} rigorously constructed \emph{vortex rings} traveling at speed $\mathcal{O}(|\log \ve|)$, thereby confirming the conjecture for translating circles; see also  \cites{MR3101789,MR302044,MR422916,MR1738349} for subsequent developments. For helical filaments, the first rigorous construction of \emph{vortex helices} was obtained in \cite{MR4417384}, validating Joukowsky’s early predictions \cite{joukowsky1912vihrevaja}, with further extensions concerning multi-helix clusters in \cite{MR4899797} and nearly parallel helical configurations in \cite{guerra2025nearlyparallelhelicalvortex}, providing a rigorous justification of the model in \cite{MR1325356}. Moreover, related works examine helical dynamics in infinite pipes and bounded helical domains \cites{donati2024dynamics,MR4555251,MR4660717,MR4534710}, including vortex patches, whereas  \cite{qin2024concentratedvortices3dincompressible} removes previous orthogonality constraints by constructing helical vortices with non-vanishing swirl. Finally, global well-posedness and long-time behaviour are addressed in \cites{MR4809294,guo2024longtimedynamicshelical,MR3320529}, while the analysis of helical structures in models closely related to the Euler equations can be found in \cites{MR3449250, MR2181460, MR4540752, MR4891431}. 
\subsection{Setting of the problem}
In this work, we revisit the vortex filament conjecture for the case of translating–rotating helices, and provide a novel construction of a smooth helical filament $\vec{\omega}_{\ve}(x,\tau)$ in the whole $\R^3$, whose \emph{cross-sectional vorticity} is compactly supported in $\R^2$ for all times. In concrete terms, for a small $\ve>0$, the three-dimensional filament will be confined within a tubular neighbourhood of size $\ve>0$ around a helix for all $\tau \in \R$, while its cross-sectional vorticity (obtained by intersecting the filament with a plane orthogonal to its tangent vector) will correspond to a smooth vorticity profile with compact support. In particular, there exist universal constants $C_1,C_2>0$ such that for all $\tau \in \R,$ it holds $C_1\ve \leq\diam\left(\operatorname{supp}\vec{\omega}_{\ve}\cap \left\{\R^2\times \{0\}\right\} \right) \leq C_2\ve$. Here and in what follows, for any bounded set $A$, we denote $\diam(A)=\sup\limits_{x,y\in A} |x-y|.$

Let us be more precise. For any $h>0$ and a fixed $r_0>0$, we consider a point $P=(r_0,0)\in\R^2$ and define the time-evolving curve  $\mathcal{G}(\tau)$ parametrised by 
\begin{equation} \label{parhelix}
	\gamma (s,\tau ) = \left( \begin{matrix}
		r_0 \cos \left(
		\frac{s-\sigma_{1} \tau}{\sqrt{h^2 + r_0^2}} \right)
		\\
		r_0 \sin \left(
		\frac{s-\sigma_1 \tau}{\sqrt{h^2 + r_0^2}} \right)
		\\
		\frac{ h s + \sigma_{2} \tau }{\sqrt{h^2+ r_0^2}}
	\end{matrix} \right) \in \R^3, \quad \sigma_1 = \frac{2\, \bar{c} \, h}{r_0^2 + h^2},
	\quad 
	\sigma_2 = \frac{2\, \bar{c}\, r_0^2}{r_0^2 + h^2},
	\end{equation}
This map gives an arclength parametrization of a circular helix with radius $r_0$ and height $h$, combining rotation and simultaneous translation. In addition, with curvature $\kappa = \frac{r_0}{r_0^2+h^2}$ and torsion $\frac{h}{r_0^2 + h^2}$, a direct computation shows that $\gamma$ evolves by the binormal flow \eqref{binormal}.
	
In view of our interest in helical filaments, we restrict our analysis to a particular class of symmetry, which naturally leads to the notion of helical symmetry. Defining the rotation matrices 
\begin{equation}\label{rotmatr}
R_{\theta} = 
\begin{pmatrix}
\cos\theta & -\sin\theta \\
\sin\theta & \cos\theta
\end{pmatrix}, \quad 
    Q_\theta = \left( \begin{matrix} \cos \theta & -\sin \theta & 0\\\sin \theta & \cos \theta & 0\\ 0 & 0 & 1\end{matrix} \right),
\end{equation}
we say that a scalar function $f:\R^3\to\R$ and a vector field $F:\R^3\to\R^3$ have helical symmetry, if for all $h>0$ and $\theta\in \R$, they satisfy 
\begin{equation}\label{helicalsym}
f\left(R_{\theta}x',x_3+h\theta\right)=f(x), \quad F\left(R_{\theta}x',x_3+h\theta\right)=Q_{\theta}F(x), \quad x=(x',x_3)\in \R^3.
\end{equation}

\subsection{Statement of the main result}
The solution we construct belongs to the class of flows with helical symmetry, subject to the additional requirement that the velocity field in \eqref{euler} is orthogonal to the tangent lines of a helix \cites{MR2505860,MR1717127}. 

The main result of this work reads as follows.
\begin{theorem}\label{maintheorem}
Let $r_0>0,\,h>0,$ and consider the helix $\mathcal{G}(\tau)$ parametrised by \eqref{parhelix}, as well as the rotation matrices $R_{\theta}$ and $Q_{\theta}$ in \eqref{rotmatr}. Then, there exist a constant $c>0$ and a smooth, global-in-time solution $\vec{\omega}_{\ve}(x,\tau)$ to \eqref{euler} with compactly supported cross-sectional vorticity in $\R^2$, such that for $\sigma_1,\,\sigma_2$ as in \eqref{parhelix} and all $\tau \in \R$, it satisfies
\begin{equation*}
\vec{\omega}_{\ve}(x,\tau)=Q_{\sigma_1\tau}\vec{\omega}_{\ve}\left(R_{-\sigma_{1} \tau}x',x_3+\sigma_{2}\tau,0\right) , \quad \vec{\omega}_{\ve}(x,\tau)=\vec{\omega}_{\ve}\left(x',x_3+2\pi h,\tau\right),
\end{equation*}
and
\[\vec{\omega}_{\ve}\left(x,|\log\ve|^{-1}\tau\right)\rightharpoonup c\delta_{\mathcal{G(\tau)}}\mathbf {T}_{\mathcal{G}(\tau)} \quad \mbox{as} \quad \ve\to 0\]
in the sense of measures, where $\delta_{\mathcal{G}(\tau)}$ denotes a uniform Dirac delta supported on $\mathcal{G}(\tau)$ and $\mathbf{T}_{\mathcal{G}(\tau)}$ its tangent unit vector. 
\end{theorem}
 
The construction of Theorem \ref{maintheorem} can be extended to a multi-vortex solution arranged along the vertices of a regular polygon. For any integer $ k \geq 2$ and the helix in \eqref{parhelix}, we consider the curves $\mathcal{G}_j(\tau)$ parametrised by 
\begin{equation}\label{helicesj}
\gamma_{j}(s,\tau)=Q_{\frac{2\pi(j-1)}{k}}\gamma(s,\tau) , \quad j=1,2,\dots,k.
\end{equation}
We additionally obtain the following Theorem.

\begin{theorem}\label{theorem 2}
Let $r_0>0,\,h>0$ and $k\geq 2$ be an integer. Consider the helices $\mathcal{G}_j(\tau), \, j=1,\dots,k,$ parametrised by \eqref{helicesj}, and the rotation matrices $R_{\theta}$ and $Q_{\theta}$ in \eqref{rotmatr}. Then, there exist a constant $c>0$ and a smooth, global-in-time solution $\vec{\omega}_{\ve}(x,\tau)$ to \eqref{euler} with cross-sectional vorticity supported in a finite union of compact balls in $\R^2$, such that for $\sigma_1,\,\sigma_2$ as in \eqref{parhelix} and all $\tau \in \R$, it satisfies
\begin{equation*}
\vec{\omega}_{\ve}(x,\tau)=Q_{\sigma_1\tau}\vec{\omega}_{\ve}\left(R_{-\sigma_{1} \tau}x',x_3+\sigma_{2}\tau,0\right) , \quad \vec{\omega}_{\ve}(x,\tau)=\vec{\omega}_{\ve}\left(x',x_3+2\pi h,\tau\right),
\end{equation*}
and
\[\vec{\omega}_{\ve}(x,|\log\ve|^{-1}\tau) \rightharpoonup c\sum_{j=1}^{k}\delta_{\mathcal{G}_j(\tau)}\mathbf{T}_{\mathcal{G}_j(\tau)}\quad \mbox{as} \quad \ve\to 0\]
in the sense of measures, where $\delta_{\mathcal{G}_j(\tau)}, \, j=1,\dots,k,$ denotes a uniform Dirac delta supported on $\mathcal{G}_j(\tau)$ and $\mathbf{T}_{\mathcal{G}_j(\tau)}$ its tangent unit vector.
\end{theorem}

We make some remarks on Theorems  \ref{maintheorem} and \ref{theorem 2}, commenting on their connection to the existing literature.
\begin{remark}
To the best of our knowledge, the construction in Theorem \ref{maintheorem} provides the first smooth helical filament in the whole space $\R^3$ with compactly supported cross-sectional vorticity in $\R^2$ (obtained by intersecting the filament with a plane orthogonal to its tangent vector) for all times. More precisely, for $\ve>0$ small, the vorticity $\vec{\omega}_{\ve}(x,\tau)$ remains confined within a tube of size $\ve>0$ around the helix \eqref{parhelix} for all $\tau \in \R$, with $C_1\ve \leq \diam{\left(\operatorname{supp}\vec{\omega}_{\ve} \cap \left\{ \R^2 \times \{0\}\right\}\right)} \leq C_2\ve$\ for some universal constants $C_1,C_2>0$. An analogous statement is also valid for Theorem \ref{theorem 2}, where the solution remains concentrated along a polygonal configuration of multiple helices for all times, with cross-sectional vorticity supported in a finite union of compact balls.
\end{remark}
\begin{remark}
A related result was obtained in \cite{donati2024dynamics} for bounded helical domains, where vorticity initially concentrated near helices of pairwise distinct radii remains concentrated, with the associated cross-sectional vorticity being supported in a thin annulus. Although the authors achieve strong radial localisation in each finite interval $[0,T],$ only $L^1$ control is obtained along the flow direction. In contrast, our solutions exhibit strong confinement in both directions, as the vorticity remains concentrated with cross-sectional vorticity compactly supported for all times, which seems unattainable in \cite{donati2024dynamics} due to the lack of uniform-in-time estimates. Besides, the assumption of distinct radii excludes polygonal configurations as in Theorem \ref{theorem 2}. Moreover, \cite{MR4534710} is concerned with related solutions in infinite pipes obtained via variational methods. However, it is unclear whether these results can be extended to the whole $\R^3$ due to the failure of the uniform ellipticity of $L_x$ in \eqref{operL} in the entire $\R^2$, and the breakdown of the  expansion of the Biot–Savart law used. Finally, for non-smooth vortex patches in the whole $\R^3$ obtained through variational methods and bifurcation theory, 
we refer to \cites{MR4834945, cao2024helicalkelvinwaves3d}.
\end{remark}

\subsection{Reduction of the problem}\label{sub13} 
We exploit the invariance of the Euler equations under helical symmetry \cites{MR2505860,MR1717127}, taking $\mathbf{u}$ and $\vec{\omega}$ to satisfy \eqref{helicalsym} and the "helical no-swirl" condition
\[\textbf{u}\cdot \vec\xi=0,\quad  \mbox{where} \quad \vec\xi(x)=(-x_2,x_1,h)\in\R^3.\]
Ettinger and Titi \cite{MR2505860} showed that solving \eqref{euler} reduces to a 2D transport equation, establishing global well-posedness through Yudovich's method \cite{MR158189}. In this framework, for $x=(x_1,x_2,x_3)\in\R^3$, $x'=(x_1,x_2)\in\R^2$ and the rescaled time $t=|\log\ve|^{-1}\tau$, we consider the scalar transport equation for $w(x',\tau)$ given by \begin{equation}\label{PB}
		\begin{aligned}
			\left\{
			\begin{aligned}
				|\log\ve|w_\tau + \nabla^\perp \psi \cdot \nabla w &=0 && {\mbox {in}} \quad \R^2 \times (-\infty, +\infty),
				\\
				- {\mbox {div}} (K \nabla \psi) &= w && {\mbox {in}} \quad \R^2 \times (-\infty,+\infty),
			\end{aligned}
			\right.
		\end{aligned}
	\end{equation}
	with the convention $(a,b)^\perp = (b,-a)$ and $K (x_1,x_2) $ denoting the matrix
\[
	K(x_1 , x_2 ) = \frac{1}{h^2+x_1^2+x_2^2}
	\left(
	\begin{matrix}
		h^2+x_2^2 & -x_1  x_2\\
		-x_1 x_2 & h^2+x_1^2
	\end{matrix}
	\right).
\]
For a self-contained derivation of \eqref{PB}, we refer to \cites{MR2505860,MR4417384}.

For $R_{\theta}$ denoting the rotation matrix in \eqref{rotmatr} and $w(x',\tau)$ solving \eqref{PB}, it was shown in \cite{MR2505860} that the vorticity vector defined as
	\begin{align}
		\label{omegaHelical}
		\vec \omega(x,\tau)  =  \frac{1}{h} w\left( R_{-\frac {x_3}h} x',\tau\right)  \, \vec\xi(x)  , \quad x = (x',x_3)\in\R^3
	\end{align}
	has helical symmetry and solves the Euler equations \eqref{euler}, while rotating solutions to \eqref{PB} of the form
\begin{equation}
\label{rotansatz}
w (x', \tau) = W \left( R_{ \alpha \tau } x' \right), \quad \psi (x',\tau) = \Psi \left( R_{ \alpha \tau } x' \right), \quad x'=(x_1,x_2)\in\R^2,
\end{equation}
can be obtained through the  semilinear elliptic equation  
\begin{equation}\label{semilin}
\nabla_{\tilde{x}}\cdot\left(K\nabla_{\tilde{x}}\Psi\right)+f\left(\Psi-\frac{\alpha}{2}|\log\ve||\tilde{x}|^2\right)=0 \quad \mbox{in} \quad \R^2, \quad \tilde{x}=R_{\alpha\tau}x'.
\end{equation}
For our objectives, one must choose a suitable nonlinear function $f$ in \eqref{semilin} so that for some $r_0>0,$ the vorticity $W(\tilde x)=f\left(\Psi-\frac{\alpha}
{2}|\log\ve||\tilde x|^2\right)$
will be compactly supported around the point $P=(r_0,0)\in\R^2$. In this way, the profile $w(x',\tau)=W(R_{\alpha\tau} x')$ in \eqref{rotansatz} will rotate rigidly around the origin with a constant rotational speed $\alpha$, yielding a three-dimensional helical filament $\vec{\omega}$ in \eqref{omegaHelical} with compactly supported cross-sections in $\R^2$ for all times.

\begin{remark}
The solutions in Theorems \ref{maintheorem} and \ref{theorem 2} are constructed using the Inner-Outer gluing scheme, which was also employed in \cite{MR4417384}. In this work, we employ a nonlinearity $f$ with compact support in \eqref{semilin}, inspired by the semilinear elliptic equation $\Delta u + u^{\gamma}_{+}=0 $ in $\R^2$, where $u_{+}=\max(0,u)$ and $\gamma>3$. In contrast to the Liouville equation $\Delta u +e^{u}=0$ in $\R^2$ used in \cite{MR4417384} where $f(s)\sim e^{s}$ yields a fast-decaying vorticity, the particular choice $f(s)\sim s^{\gamma}_{+}$ gives vorticity supported within a tube of size $\ve>0$ around the helix for all times. We note that a similar nonlinearity has been used in \cites{davila2023global,MR4767454} to construct traveling vortex-pair solutions for the Euler equations in $\R^2$, while a nonlocal analogue was employed in \cite{MR4302173} to obtain traveling and rotating solutions for the generalized inviscid SQG equation in $\R^2$. Nevertheless, these results cannot be adapted to our setting, since they concern bounded solutions (cf. Lemma \ref{lemat}).
\end{remark}
\begin{remark}
    Our constructions rely on elliptic singular perturbation methods, where a crucial step involves the linearisation around an approximate solution. In this framework, we are naturally led to study the linearised operator $\Delta+\gamma\Gamma^{\gamma-1}_{+}$ in $\R^2$, where $\Gamma$ is defined in \eqref{defGamma}. A non-degeneracy result of Dancer and Yan \cite{MR2456896} indicates that the only bounded functions in the kernel of this operator in $\R^2$ are given by the partial derivatives of $\Gamma$. However, in our setting, the solution exhibits logarithmic growth at infinity (see Lemma \ref{lemat}), which introduces an additional radial component $Z_0=\mathcal{O}\left(\left(\log(2+|y|\right)\right)$ in the kernel. This feature complicates our analysis, since the estimate $\int_{\R^2}\gamma\Gamma^{\gamma-1}_{+}Z_0 = O(1)$ induces a strong coupling between the Inner and Outer Equations (we refer to Section \ref{Section 5} and Proposition \ref{propsec7} for more details). 
    On the other hand, \cite{MR4417384} shows that the associated radial element $\tilde{Z}_{0}=O(1)$ in the kernel of the linear operator $\Delta+e^{u}$ in $\R^2$   (with u solving $\Delta u+e^{u}=0$ in $\R^2$) satisfies
$\int_{\R^2} e^{u}\,\tilde Z_{0}=0$, which decouples the Inner-Outer system at main order. This obstruction implies that the construction of the helical filaments in our case is more delicate, requiring several adjustments of existing gluing methods.
\end{remark}
In summary, following the preceding discussion, in the remainder of this work we write $x$ instead of $\tilde{x}$ in \eqref{semilin} for ease of notation, and focus on solving the equation
\begin{equation}\label{eqtosolve}
\nabla_{x}\cdot\left(K\nabla_{x}\Psi\right)+f\left(\Psi-\frac{\alpha}{2}|\log\ve||x|^2\right)=0 \quad \mbox{in} \quad \R^2.
\end{equation}
 For a fixed $r_0>0$ and the point $P=(r_0,0)\in\R^2,$ our aim is to construct a stream function $\Psi$ for \eqref{eqtosolve} such that for some constant $c>0$ it exhibits the asymptotic behaviour $-\nabla_x \cdot (K\nabla_x\Psi) \sim c\delta_{P}$ as $\ve\to 0$, while we also require that the vorticity profile $W(x)=f\left(\Psi-\frac{\alpha}{2}|\log\ve||x|^2\right)$   has compact support around $P.$ Therefore, we realise that in a small neighbourhood of the point $P,$\, the stream function $\Psi$ must resemble a Green's
function for the elliptic operator $-\nabla_{x}\cdot(K\nabla_{x}\cdot)$ in $\R^2$ up to lower-order corrections, capturing the locally singular structure. 
\subsection{Notation} Throughout the paper,  $c>0$ and $C>0$ denote generic constants that may change from line to line and are independent of all relevant quantities. In addition,  $B_{r}=B_{r}(0)$ denotes the open ball of radius $r>0$ centered at the origin.

\subsection{Structure of the paper}
This work is organized as follows. In Section \ref{Section 2}, we construct a regularized approximation of the Green’s function for the operator $-\nabla_{x}\cdot(K\nabla_{x}\cdot)$ in $\R^2$ via a change of variables, using a radial profile $\Gamma$ solving a semilinear elliptic equation. Section \ref{Section 3} selects a compactly supported nonlinearity $f$ and derives error estimates for the preceding approximation. In Section \ref{Section 4}, we apply the \emph{Inner–Outer gluing scheme} to perturb the approximation into an exact solution, relying on the linear theories from Sections \ref{Section 5} and \ref{Section 6}. Section \ref{Section projected} establishes existence and uniqueness for a projected linearised problem through a fixed point argument. In Section \ref{reduced}, we choose the rotational speed in \eqref{rotansatz} so that the solution of the projected problem can be lifted to a genuine solution. Finally, Section \ref{multihelices} outlines the proof of Theorem \ref{theorem 2}.
\section{Construction of the approximate stream function}\label{Section 2} 
Let $r_0>0,\, h>0 $ and consider the point $P=(r_0,0)\in\R^2$. This section is devoted to the construction of an approximate  stream function $\Psi$ satisfying 
\begin{equation}\label{ori}
-\nabla_{x}\cdot (K\nabla_{x}\Psi) \sim c\delta_P
\end{equation}
in a neighbourhood of the point $P$. 

Hereafter, for ease of notation, we set 
\begin{equation}\label{operL}
L_x=\nabla_{x}\cdot (K\nabla_{x}\cdot),
\quad \text{with} \quad
K=\frac{1}{h^2+x_1^2+x_2^2}
\begin{pmatrix}
h^2+x_2^2 & -x_1x_2 \\
-x_1x_2 & h^2+x_1^2
\end{pmatrix}.
\end{equation}
We construct an approximate Green's function for \eqref{ori} by studying $L_x$ locally around $P=(r_0,0)\in\R^2$. Unlike the Laplacian $\Delta_x$, this elliptic operator in divergence form has variable coefficients, producing anisotropic diffusion that degenerates for large $|x|$. Besides, the parameter $h$ formally interpolates between the 3D axisymmetric case ($h\to 0$) and the standard planar setting ($h\to\infty$).

Unlike planar or axisymmetric settings, the helical framework lacks coordinates aligned with symmetry directions \cite{MR2505860}. We therefore introduce the change of variables
\begin{equation}\label{var}
x-P=A[P]z,\quad
A[P]=
\begin{pmatrix}
\frac{h}{\sqrt{h^2+r_0^2}}
& 0\\
0
& 1
\end{pmatrix},
\end{equation}
originally introduced in \cite{MR4899797}. This anisotropic scaling normalizes the coefficient matrix $K$ in \eqref{operL} at $P=(r_0,0)$, recasting $L_x$ as $\Delta_z$ plus lower-order corrections in the $z$ variable, thus simplifying the analysis of \eqref{ori}.

This is the content of the next Proposition.
\begin{prop}\label{expansionofL}
Let $z$ be the variable in \eqref{var}. It holds 
\[L_{x}=\Delta_{z}+B_0,\]
where
\begin{align*}
B_{0}
&= \left(\frac{h^2(r_0^2-r^2)+z_2^2(h^2+r_0^2)}{h^2(h^2+r^2)}\right)\partial_{z_1z_1}
+\frac1{(h^2+r^2)}\left(\left(z_1\frac h{\sqrt{h^2+r_0^2}}+r_0\right)^2-r^2\right)\partial_{z_2z_2} \nonumber\\
 &- 2\frac{\sqrt{h^2+r_0^2}}{h(h^2+r^2)}z_2\left(z_1\frac{h}{\sqrt{h^2+r_0^2}}+r_0\right)
\partial_{z_1z_2}
 \nonumber\\
& - \frac{z_1(h^2+r_0^2)+r_0h\sqrt{h^2+r_0^2}}{h^2(h^2+r^2)}\left(1+\frac{2h^2}{h^2+r^2}\right)
\partial_{z_1}
 - \frac{z_2}{h^2+r^2}\left(1+\frac{2h^2}{h^2+r^2}\right)\partial_{z_2},
\end{align*}
and
\begin{equation}\nonumber
r^2=|x|^2=r_0^2+2r_0\frac h{\sqrt{h^2+r_0^2}}z_1 + \frac{h^2}{h^2+r_0^2}z_1^2+z_2^2.
\end{equation}
\end{prop}

\begin{proof}
The result follows from setting $(a,b)=(r_0,0)$ in (\cite{MR4899797}, Proposition 2.1).
\end{proof}

Due to Proposition \ref{expansionofL}, we realise that in the region $|z|<\delta$ for a fixed small $\delta > 0$, the operator $L_x$ is a perturbation of the usual Laplace operator in $\R^2$ given by
\begin{equation}\label{oper}
L_x=\Delta_z+B_{0},
\end{equation}
where $B_{0}$ can be expressed as
\be
\begin{aligned}\nonumber
B_0
&= \left(-2\frac{r_0h}{(h^2+r_0^2)^{\frac{3}{2}}}z_1+O(|z|^2)\right)\partial_{z_1z_1}
+\mathcal{O}(|z|^2)\partial_{z_2z_2} -\left(2\frac {r_0}{h\sqrt{h^2+r_0^2}}z_2+\mathcal{O}(|z|^2)\right)\partial_{z_1z_2}  \nonumber \\
&-\left(\frac{r_0}{h\sqrt{h^2+r_0^2}}\left(1+\frac{2h^2}{h^2+r_0^2}\right)+\mathcal{O}(|z|)\right)
\partial_{z_1} -\left(\frac{z_2}{h^2+r_0^2}\Big(1+\frac{2h^2}{h^2+r_0^2}\right)
+\mathcal{O}\left(|z|^2)\right)
\partial_{z_2},
\end{aligned}
\ee
with no constant multiples of second-order derivatives appearing. 

In addition, (\ref{ori}) is now equivalent to
\begin{equation}\label{oriz}
 -(\Delta_z+B_0)\psi = c\delta_{0} \quad\mbox{and}\quad \psi(z)=\Psi(P+A[P]z),
\end{equation}
where $\delta_{0}$ is a Dirac delta centered at the origin, and $A[P]$ is the matrix  associated with the change of variables in \eqref{var}. From this perspective, obtaining an approximate regularisation of the Green's function of the operator  $L_x$ in $\R^2$ leads to constructing a regularisation of the Green's function of the Laplace operator $\Delta_z$ in $\R^2$.
\medskip 

Accordingly, we now introduce the regularising profile associated with $\Delta_z$ in $\R^2$ as follows. For any $\gamma>3$ and $s_{+}=\max\{0,s\}$, consider the semilinear elliptic equation
\begin{equation}\label{limit}
\Delta_z\Gamma + \Gamma_+^\gamma=0 \quad \mbox{on}\ \mathbb{R}^2,
\quad \{\Gamma>0\}= B_1(0),
\end{equation}
which admits a radially symmetric, classical solution of the form
\begin{equation}\label{defGamma}
\begin{aligned}
\Gamma(z)
  =
\left\{\begin{array}{ll}
\nu(|z|)&\mathrm{if~}|z|\leq1\\
\nu'(1)\log|z|&\mathrm{if~}|z|>1
\end{array}
\right.,\\
\end{aligned}
\end{equation}
 where $\nu$ denotes the unique, radial ground state solution of
 \[
 \Delta_z
 \nu+\nu^\gamma =0,\quad \nu\in H_0^1(B_1),\quad \nu>0\,\text{ on }B_1.
 \]
Due to Hopf's lemma we get $\nu'(1)<0$, which will be an important feature for the analysis carried out in Section \ref{Section 3}. Moreover, for any $\ve >0$, it is easy to see that the rescaled profile 
\begin{equation}\label{erescgamma}
\Gamma_{\ve}(z)=\Gamma\left(\frac{z}{\ve}\right)
\end{equation}
is $C^{1}$ on the boundary of the ball $B_{\ve}(0)$ and satisfies the rescaled semilinear elliptic equation
\begin{equation}\label{rescaledequation123}
    \Delta_z\Gamma_{\ve}(z)+\frac{1}{\ve^{2}}\left(\Gamma_{\ve}(z)\right)_+^\gamma=0
\quad  \mbox{in}\ \mathbb{R}^2.
\end{equation}
In what follows, we restrict our attention to the region $|z|<\delta$ for a small $\delta>0$, and define
\begin{equation}\label{defofgammahat}
\hat{\Gamma}_{\ve}(z)=\Gamma\left(\frac{z}{\ve}\right)-\nu'(1)|\log\ve|,
\end{equation}
which satisfies \[-\Delta_z\hat{\Gamma}_{\ve}(z) \rightharpoonup c\delta_0  \quad  \mbox{as} \quad \ve \to 0, \quad \mbox{where} \quad c=\int_{\R^2} \Gamma^{\gamma}_{+}.\]
Motivated by the semilinear equation \eqref{limit}, we use \eqref{defofgammahat} as the fundamental building block and employ elliptic singular perturbation methods to construct an approximate solution of \eqref{oriz}.

The following Proposition holds.

\begin{prop}\label{propapprox}
Given $r_0>0$ and $h>0,$ consider the  profile $\hat{\Gamma}_{\ve}(z)$ in \eqref{defofgammahat} and define the approximate solution of \eqref{oriz}
as \[\psi_{3\ve}(z)= \frac{\alpha}{2}|\log\ve|r_0^2 +\hat{\Gamma}_{\ve}(z)\left(1+c_1z_1+c_2|z|^2\right) +\frac{r_0^3}{2h(h^2+r_0^2)^{\frac{3}{2}}}H_{1\ve}(z),\]
where \[c_1=\frac{r_0h}{2(h^2+r_0^2)^{\frac{3}{2}}}, \quad c_2=\frac{3h^2r_0^2+r_0^4}{8(h^2+r_0^2)^{3}},\]
and \[\Delta_z H_{1\ve} +\frac{\operatorname{Re}(z^3)}{\ve^2|z|^2}\left(\Gamma''\left(\frac{z}{\ve}\right)-\frac{\Gamma'\left(\frac{z}{\ve}\right)}{\frac{z}{\ve}}\right)=0.\]
It the rescaled variable $y=\frac{z}{\ve}$ it holds $\psi_{3\ve}(\ve y_1,\ve y_2)=\psi_{3\ve}(\ve y_1,-\ve y_2),$ while for any small $\delta>0$, in the region $|y| < \frac{\delta}{\ve}$ we have
\[\ve^2L_{x}(\psi_{3\ve})(\ve y)= \Delta_y\Gamma(y) + \frac{3r_0h^2+r_0^3}{2h(h^2+r_0)^{\frac{3}{2}}}\ve y_1\left(\Gamma(y)\right)^{\gamma}_{+} + \ve^2 E_{*},\]
where $E_{*}$ is a smooth function in $z=\ve y$, uniformly bounded as $\ve \to 0.$

In the original variable $x,$ the approximate regularisation can be expressed as \[\Psi_{P}(x)=\psi_{3\ve}\left(A[P]^{-1}\left(x-P\right)\right),\]
where $P=(r_0,0)\in\R^2$ and $A[P]$ is given in \eqref{var}.
\end{prop}
\begin{proof}
Let $\delta>0$ be small. We introduce the rescaled variable $y=\frac{z}{\ve},$ consider the region $|y|<\frac {\delta}{\ve},$ and use \eqref{oper} in order to compute
\begin{equation}\label{Bfirsterror}
\begin{aligned}
\ve^2 B_{0}(\ve y)[\hat{\Gamma}_{\ve}]
=& -\frac{2r_0 h}{(h^2+r_0^2)^{\frac 3 2}}\ve y_1\partial_{y_1y_1}\hat{\Gamma}_{\ve}-\frac{2r_0}{h\sqrt{h^2+r_0^2}}\ve y_2\partial_{y_1y_2}\hat{\Gamma}_{\ve}\\
&- \frac{ r_0}{h\sqrt{h^2+r_0^2}}\left(1+\frac{2h^2}{h^2+r_0^2}\right)\ve
\partial_{y_1}\hat{\Gamma}_{\ve}
+\ve^2 E_1,
\end{aligned}
\end{equation}
where $E_1$ is a smooth function in $z=\ve y$, uniformly bounded as  $\ve\to 0$.

Since $\hat\Gamma_{\ve}$ in \eqref{defofgammahat} is radial  in $y$, we then calculate
\begin{equation}\nonumber
\partial_{y_1} \hat\Gamma_{\ve}=\Gamma' \frac{y_1}{|y|},  \quad
\partial_{y_1y_1}\hat\Gamma_{\ve}
= \left(\Gamma''-\frac{\Gamma'}{|y|}\right)\frac{y_1^2}{|y|^2}+\frac{\Gamma'}{|y|}, \quad \partial_{y_1 y_2}\hat\Gamma_{\ve}=\left(\Gamma''-\frac{\Gamma'}{|y|}\right)\frac{y_1 y_2}{|y|^2},
\end{equation}
where $\Gamma'$ designates the radial derivative of $\Gamma(y)=\Gamma_{\ve}(z)$.

 Substituting the above into \eqref{Bfirsterror} and identifying the vector $y=(y_1,y_2)\in \R^2$ with the complex number $y=y_1+iy_2 \in \mathbb{C}$, by virtue of the identities 
 \[y_1y_2^2=\frac{y_1|y|^2}{4}-\frac{\operatorname{Re}(y^3)}{4}, \quad y_1^3=\frac{3y_1|y|^2}{4}+\frac{\operatorname{Re}(y^3)}{4},\] 
 we further find 
\begin{equation}\label{bgamma2}
\begin{aligned}
\ve^2 B_{0}(\ve y)[\hat\Gamma_{\ve}]=&-\frac{r_0h}{(h^2+r_0^2)^{\frac{3}{2}}}\frac{\Gamma'\ve y_1}{|y|}-\frac{r_0^3+4r_0h^2}{2h(h^2+r_0^2)^{\frac{3}{2}}}\left(\Gamma''+\frac{\Gamma'}{|y|}\right)\ve y_1 \\
&+\frac{\ve r_0^3}{2h(h^2+r_0^2)^{\frac{3}{2}}}\left(\Gamma''-\frac{\Gamma'}{|y|}\right) \frac{\operatorname{Re}(y^3)}{|y|^2} +\ve^2 E_1,
\end{aligned}
\end{equation}
with $\operatorname{Re}(y^3)$ denoting the real part of $y^3$ and $E_1$ as in \eqref{Bfirsterror}. 

To proceed, we modify the approximation to eliminate the first error term in \eqref{bgamma2}. Defining \[
\psi_{1\ve}(\ve y)=\frac{\alpha}{2}|\log\ve|r_0^2 + (1+c_1\ve y_1)\hat{\Gamma}_{\ve},\]
we observe that
\[\Delta _{y}(c_1\ve y_1\hat{\Gamma}_{\ve})
   =c_1\ve y_1\Delta_{y}\Gamma +2c_1 \frac{\Gamma'\ve y_1}{|y|},\]
hence choosing
$    c_1=\frac{1}{2}\frac{r_0 h}{(h^2+r_0^2)^{\frac 3 2}}
$
yields
\begin{align*}
\ve^2 L_{x}\left(\psi_{1\ve}\right)
&=\Delta_{y}\Gamma
-\left(\frac{3r_0h^2+r_0^3}{2h(h^2+r_0^2)^{\frac{3}{2}}}\right)
 \left(\Gamma'' +\frac{\Gamma'}{|y|}\right)\ve y_1
+ \frac{\ve r_0^3}{2h(h^2+r_0^2)^{\frac{3}{2}}}
\left(\Gamma'' -\frac{\Gamma'}{|y|}\right)\frac{\operatorname{Re}(y^3) }{ |y|^2}\nonumber \\
&  + \ve^2B_{0}(\ve y)\left[c_1\ve y_1\hat{\Gamma}_{\ve}\right]+\ve^2 E_1.
\end{align*} 
Following a similar reasoning as before, one can show that
\[\ve^2
B_0(\ve y)[c_1\ve y_1 \hat\Gamma_{\ve}]= -\frac{c_1\ve^2 r_0}{h\sqrt{h^2+r_0^2}}\left(1+\frac{2h^2}{h^2+r_0^2}\right)\hat{\Gamma}_{\ve}+\ve^2 E_2,\]
where $E_{2}$ is another smooth function in $\ve y$ which is uniformly bounded as $\ve \to 0,$ while
\begin{equation}\label{jae}
\begin{aligned}
\ve^2 L_{x}\left(\psi_{1\ve}\right)
&=  \Delta_{y}\Gamma
-\frac{3r_0h^2+r_0^3}{2h(h^2+r_0^2)^{\frac{3}{2}}}\left(\Gamma'' +\frac{\Gamma'}{|y|}\right)\ve y_1
-
 \frac{\ve^2 (3h^2r_0^2+r_0^4)}{2(h^2+r_0^2)^3} \hat{\Gamma}_{\ve}
\\
&  + \frac{\ve r_0^3}{2h(h^2+r_0^2)^{\frac{3}{2}}}
\left(\Gamma^{''} -\frac{\Gamma'}{|y|}\right)\frac{\operatorname{Re}(y^3) }{|y|^2}
+\ve^2(E_1+E_2).
\end{aligned}
\end{equation}
The next step for improving the approximation is to eliminate the third error term in \eqref{jae}, which is proportional to $\hat{\Gamma}_{\ve}.$ 
We notice that 
\[
   \Delta_{y}\left(c_2 \ve^2|y|^2\hat\Gamma_{\ve}\right)
=c_2\ve^2\left(4\hat{\Gamma}_{\ve}+4 y\cdot\nabla_{y}\Gamma+|y|^2\Delta_{y}\Gamma\right),\]
thus we adjust the approximation by setting 
\begin{equation*}
\psi_{2\ve}(\ve y)=\frac{\alpha}{2}|\log\ve|r_0^2+\left(1+c_1\ve y_1+c_2\ve^2|y|^2\right)\hat\Gamma_{\ve}, \quad c_2= \frac 1 8 \frac{3h^2r_0^2+r_0^4}{(h^2+r_0^2)^3}.
\end{equation*}
With this modification, we infer that the approximate regularisation $\psi_{2\ve}$ satisfies
\begin{equation}\label{errorsecondapprox}
\begin{aligned}
\ve^2 L_{x}\left(\psi_{2\ve}\right) &=  \Delta_{y}\Gamma
-\frac{3r_0h^2+r_0^3}{2h(h^2+r_0^2)^{\frac 3 2}}
 \left(\Gamma'' +\frac{\Gamma'}{|y|}\right)\ve y_1
  + \frac{\ve r_0^3}{2h(h^2+r_0^2)^{\frac 3 2}} \left(\Gamma'' -\frac{\Gamma'}{|y|}\right)\frac{\operatorname{Re}(y^3)}{|y|^2}\\
  &+\ve^2(E_1 +E_2+E_3),
\end{aligned}
\end{equation}
where $E_3$ has the same properties as $E_1$ and $E_2.$

To cancel the third term in (\ref{errorsecondapprox}), we introduce a radial correction $h_{1\ve}(s)$ defined as the solution of
\[h''_{1\ve}+\frac{1}{s} h'_{1\ve}-\frac{9}{s^2}h_{1\ve} +\frac{s}{\ve^2}\left(\Gamma''\left(\frac{s}{\ve}\right)-\frac{\Gamma'\left(\frac{s}{\ve}\right)}{\frac{s}{\ve}}\right)=0,\]
which can be expressed as
\[
h_{1\ve}(s)
= s^3\int_s^1\frac{\rm dr}{r^7}
\int_0^r \frac{t^5}{\ve^2}\left(\Gamma''\left(\frac{t}{\ve}\right) -\frac{\Gamma'\left(\frac{t}{\ve}\right)}{\frac{t}{\ve}}\right) \rm d t.
\]
We remark that $\tilde{F}_{\ve}(s)\defeq\frac{s}{\ve^2}\left(\Gamma''\left(\frac{s}{\ve}\right)-\frac{\Gamma'\left(\frac{s}{\ve}\right)}{\frac{s}{\ve}}\right)$ is uniformly bounded as $\ve\to 0.$ 
To see this, since $\ve \to 0$ corresponds to $u\defeq \frac{s}{\ve} \to \infty$, using the definition of $\Gamma$ in \eqref{defGamma} one can find that $u\left(\Gamma''(u)-\frac{\Gamma'(u)}{u}\right)=\mathcal{O}\left(\frac{1}{u}\right)$ as $u \to \infty.$ In addition, we observe that $\lim\limits_{u\to 0}\frac{\Gamma'(u)}{u}=\Gamma''(0),$ which in turn yields $\left(\Gamma''(u)-\frac{\Gamma'(u)}{u}\right)=\mathcal{O}(u)$ as $u \to 0.$ Combining these asymptotics, it can be verified that $h_{1\ve}$ is a smooth function in $s$, uniformly bounded as $\ve \to 0$, satisfying $h_{1\ve}(s)=\mathcal{O}(s^3)$ as $s\to 0$. 

Identifying again the vector $(y_1,y_2)\in\R^2$ with $y_1+iy_2 \in\C$ and turning to polar coordinates $y=|y|e^{i\theta}$, the function $H_{1\ve}(z)=h_{\ve}(|z|)\cos(3\theta)$ satisfies the equation
\begin{equation*}
\Delta_{y} H_{1\ve}+\frac {\ve \operatorname{Re}(y^3)}{|y|^2} \left(\Gamma''(y) -\frac{\Gamma'(y)}{|y|}\right)=0,
\end{equation*}
thus we refine the approximation of the stream function as
\begin{equation}\label{psi3}
\psi_{3\ve}(\ve y)=\frac{\alpha}{2}|\log\ve|r_0^2 + \left(1+c_1\ve y_1+c_2\ve^2|y|^2\right)\hat\Gamma_{\ve}
+\frac{r_0^3}{2h(h^2+r_0^2)^{\frac{3}{2}}}H_{1\ve}(\ve y).
\end{equation}
Combining the results above, straightforward computations show that 
\begin{equation}
\begin{aligned}\label{jadse}
\ve^2 L_{x}\left(\psi_{3\ve}\right)= &\Delta_y\Gamma
+\frac{3r_0h^2+r_0^3}{2h(h^2+r_0)^\frac{3}{2}}
 \ve y_1 \Gamma_{+}^{\gamma}
 +\ve^2 E_{*},
\end{aligned}
\end{equation}
where $E_{*}$ is once again a smooth function in $\ve y$, uniformly bounded as $\ve \to 0$. 

Referring to the change of variables in \eqref{var}, we can express the approximate stream function (\ref{psi3})  around $P=(r_0,0)\in\R^2$ as
\begin{equation*}
\Psi_{P}(x)=\psi_{3\ve}\left(A[P]^{-1}(x-P)\right),
\end{equation*}
which concludes the proof.
\end{proof}

\subsection{Globally defined approximate regularisation}
At this point, in order to extend the definition of the approximate stream function of Proposition \ref{propapprox} to the whole $\R^2$, we introduce a smooth cut-off function 
\begin{equation}\label{definitioneta}
\eta_{\delta}(x)=\eta\left(\frac{|z|}{\delta}\right), \quad \text{where} \quad \eta(s)=1 \quad \text{for} \quad |s|\leq \frac{1}{2}, \quad \eta(s)=0 \quad \text{for} \quad |s|\geq 1,
\end{equation}
and consider the modified approximation 
\begin{equation}\label{globalapprox1}
    \tilde{\Psi}_{\alpha}(x)=\frac{\alpha}{2}|\log\ve|r_0^2+\eta_{\delta}(x)\left(\left(1+c_{1}z_{1}+c_{2}|z|^2\right)\hat\Gamma_{\ve}(z)+\frac{r_0^3}{2h(h^2+r_0^2)^{\frac{3}{2}}}H_{1\varepsilon}(z)\right).
\end{equation}
Since the operator $L_{x}$ in \eqref{operL} is linear, it follows from \eqref{psi3} and \eqref{globalapprox1} that
\begin{equation}\label{errorg}
        L_{x}(\tilde{\Psi}_{\alpha}) =\eta_{\delta}(x)\left(L_{x}\left(\psi_{3\ve}\right)-E_{*}\right)+g(x),
\end{equation}
with $E_{*}$ as in \eqref{jadse} and
\[g(x)=\eta_{\delta}(x)E_{*}+L_{x}\left(\eta_{\delta}(x)\left( \psi_{3\ve}-\frac{\alpha}{2}|\log\ve|r_0^2\right)\right)-\eta_{\delta}(x)L_{x}\left(\psi_{3\ve}-\frac{\alpha}{2}|\log\ve|r_0^2\right).\] 
One can verify that $g(x)$ has compact support, thus we readily get  $\|g\|_{L^{\infty}(\R^2)}\leq C_{\delta}$ for a constant depending on $\delta>0$.

The next step towards a global improvement of the approximation \eqref{globalapprox1} is to add a correction $H_{2\ve}(x)$ in the original variable $x$, to cancel the bounded error term $g(x)$ in \eqref{errorg}. To be precise, we let $H_{2\ve}(x)$ satisfy the Poisson equation
\begin{equation}\label{eqforh2e}
L_{x}(H_{2\varepsilon})+g(x)=0 \quad \text{in} \quad \R^2,
\end{equation}
where an application of Proposition \ref{propouter} guarantees a positive solution to \eqref{eqforh2e} satisfying the bound 
\begin{equation}\label{boundh2e}
|H_{2\ve}|\leq C_{\delta}\left(1+|x|^2\right).
\end{equation}
Moreover, as the solution is determined only up to an additive constant, we select the particular one satisfying 
\begin{equation}\label{H2vanishing}
    H_{2\varepsilon}(r_0,0)=0.
\end{equation}
In terms of the variable $z=\ve y$ and \eqref{var}, the final global approximation reads 
\begin{equation}\label{globalapprox}
 \Psi_{\alpha}(x)=\frac{\alpha}{2}|\log\ve|r_0^2+\eta_{\delta}(x)\left(\left(1+c_{1}z_{1}+c_{2}|z|^2\right)\hat\Gamma_{\ve}(z)+\frac{r_0^3}{2h(h^2+r_0^2)^{\frac{3}{2}}}H_{1\ve}(z)\right)+H_{2\ve}(x),
\end{equation}
with $\Psi_{\alpha}$ being even in $x_2$, i.e. $\Psi_{\alpha}(x_1,-x_2)=\Psi_{\alpha}(x_1,x_2).$ This property holds since $\psi_{3\ve}$ in \eqref{psi3} is even in $x_2=\ve y_2$, while $H_{2\ve}(x)$ inherits the same symmetry due to $g(x)$ being even in $x_2$.

Finally, due to Proposition \ref{propapprox}, the action of the operator $L_{x}$ to the final approximation $\Psi_{\alpha}$ in \eqref{globalapprox} is equal to 
\begin{equation}\label{Lerror}
L_{x}(\Psi_\alpha)=\eta\left({\frac{|z|}{\delta}}\right)\left(\Delta_{z}\Gamma_{\ve}+\frac{3r_0h^2+r_0^3}{2h(h^2+r_0^2)^{\frac{3}{2}}}\frac{z_1}{\ve ^2}\left(\Gamma_{\ve}\right)^{\gamma}_{+}\right),
\end{equation}
with $\Gamma_{\ve}(z)$ as in \eqref{erescgamma}.
\section{Choice of Nonlinearity and Error of Approximation}\label{Section 3}
In this section, we select a nonlinearity $f$ and estimate the error when evaluating \eqref{eqtosolve} at $\Psi_{\alpha}$ from \eqref{globalapprox}. We define the error operator
\begin{equation}\label{definitionerroroper}
S(\Psi)(x)=L_{x}(\Psi)+f\left(\Psi-\frac{\alpha}{2}|\log\ve||x|^2\right),
\end{equation}
where exact solutions satisfy $S(\Psi)=0$. As discussed in Section \ref{sub13}, we choose $f$ so that the vorticity $W(x)=f\left(\Psi-\tfrac{\alpha}{2}|\log\ve||x|^2\right)$ concentrates near $P=(r_0,0)$, with $W \rightharpoonup c\delta_P$ as $\ve \to 0$, yielding a 3D helical filament with compactly supported cross-section in $\R^2$.

Motivated by \eqref{rescaledequation123}, we introduce the nonlinearity 
\begin{equation}\label{nonlinearity}
f(s)=\frac{1}{\ve^2}\left(s+\nu'(1)|\log\ve|\right)^{\gamma}_{+},
\end{equation}
where $\gamma >3$ and $s_{+}=\max(s,0)$. One can verify that this nonlinearity has sufficient regularity for the construction of a smooth vorticity, since $s\mapsto s^{\gamma}_{+}\in C^{\gamma-1}$ if $\gamma\in\mathbb{Z}$ and $s\mapsto s^{\gamma}_{+}\in C^{\lfloor{\gamma}\rfloor,\gamma-\lfloor{\gamma}\rfloor}$ otherwise.

In the sequel, for given $r_0>0$, $h>0$, and recalling that $\nu'(1)<0,$ we fix the rotational speed in \eqref{rotansatz} as
\begin{equation}\label{rotspeed}
\alpha=-\frac{\nu'(1)}{2\left(h^2+r_0^2\right)} + \tilde{\alpha},\quad \mbox{where} \quad  \tilde{\alpha}=\mathcal{O}\left(|\log\ve|^{-1}\right).
\end{equation}
The rigorous justification of this particular choice will be carried out in Section \ref{reduced}, while currently our aim is to derive error estimates for $S\left(\Psi_{\alpha}\right)$.

We have the following Proposition.
\begin{prop}\label{errorprop1}
    Consider $r_0>0,\, h>0$, and let $\alpha$ be the rotational speed in \eqref{rotspeed}. For the approximate stream function $\Psi_{\alpha}$ in \eqref{globalapprox} and the function $f$ in \eqref{nonlinearity}, it holds that for sufficiently small $\delta>0$, there exists a constant $C>0$ such that for all $\ve >0$ small, the error function in \eqref{definitionerroroper} satisfies
    \[\ve^2 S\left(\Psi_{\alpha}\right)\leq C\ve|y|\Gamma^{\gamma}_{+},\]
where $y=\frac{z}{\ve}$ and $\operatorname{supp} \Gamma^{\gamma}_{+} \subset B_1 (0).$
\end{prop}
\begin{proof}
To estimate the error of approximation \[\ve^2 S(\Psi_{\alpha})= \ve^2 L_{x}(\Psi_{\alpha})+\left(\Psi_{\alpha}-\frac{\alpha}{2}|\log\ve||x|^2+\nu'(1)|\log\ve|\right)^{\gamma}_{+},\]
it is necessary to consider three different regions in the $y=\frac{z}{\ve}$ variable as follows.
\\
\text{\bf Case 1:} $|y|\leq 1.$
\\
Due to \eqref{definitioneta}, in this region we have $\eta_{\delta}=1,$ thus using \eqref{Lerror} we find
\begin{equation*}
    \begin{aligned}
    \ve^2 S\left(\Psi_{\alpha}\right)&=\Delta_{y}\Gamma+\left(\frac{3r_0 h^2+r_0^3}{2h\left(h^2+r_0^2\right)^{\frac 3 2}}\right)\ve y_1 \Gamma^{\gamma}_{+}\\
&+\Bigg[\frac{\alpha}{2}|\log\ve| r_0^2+\left(\Gamma(y)-\nu'(1)|\log\ve|\right)\left(1+c_1\ve y_1 +c_2 \ve^2 |y|^2\right)+\frac{r_0^3}{2h\left(h^2+r_0^2\right)^{\frac{3}{2}}}H_{1\ve}(\ve y)\\
&\hspace{5mm}+H_{2\ve}(x)-\frac{\alpha}{2}|\log\ve||x|^2+\nu'(1)|\log\ve|\Bigg]^{\gamma}_{+}.
\end{aligned}
\end{equation*}
Using the asymptotic expansions 
\begin{equation}\label{Expcase1}
    H_{2\ve}(x)=H_{2\ve}(P)+\ve\left(A[P]y\right)\cdot\nabla H_{2\ve}(P) +\mathcal{O}\left(\ve^2 |y|^2\right), \quad  H_{1\ve}=\mathcal{O}(\ve^3 |y|^3) \quad \mbox{as} \quad \ve\to 0,
\end{equation}
a Taylor approximation around $\Gamma(y)$ yields 
\begin{equation*}
\begin{aligned}
   \ve^2 S\left(\Psi_{\alpha}\right)&=\Delta_{y}\Gamma + \left(\frac{3r_0h^2+r_0^3}{2h\left(h^2+r_0^2\right)^{\frac{3}{2}}}\right)\ve y_1\Gamma^{\gamma}_{+}+\Gamma^{\gamma}_{+} +\gamma\Gamma^{\gamma-1}_{+}\Bigg[\ve y_1|\log\ve|\left(-c_1\nu'(1)-\frac{\alpha r_0h}{\sqrt{h^2+r_0^2}}\right)\\
   &+\ve y_1\left(c_1\Gamma(y)+\frac{h}{\sqrt{h^2+r_0^2}}\partial_{y_1}H_{2\ve}(P)\right)+\mathcal{O}\left(\ve^2 |\log\ve| |y|^2\right)\Bigg] +\mathcal{O}\left(\gamma(\gamma-1)\Gamma^{\gamma-2}_{+}\ve^2|y|^2\right),
    \end{aligned}
\end{equation*}
where we used \eqref{H2vanishing}, the fact that $H_{2\ve}(x)$ is even in $x_2$ and \(|x|^2=r_0^2+2r_0\frac{h}{\sqrt{h^2+r_0^2}}\ve y_1 +\frac{h^2}{h^2+r_0^2}\ve^2 y_1^2 +\ve^2y_2^2.\)

Employing \eqref{limit} and \eqref{rotspeed}, one can directly verify that 
\begin{equation}\label{errorinner1}
\ve^2 S(\Psi_{\alpha})(\ve y)=\mathcal{O}\left(\ve |y|\Gamma^{\gamma}_{+}\right).
\end{equation}
\text{\bf Case 2:} $1 < |y| < \frac{\delta}{\ve}.$ 
\\
In this intermediate region we have that $\eta_{\delta}(x)\in(0,1]$ and $\Gamma(y)=\nu'(1)\log|y|<0$, thus from \eqref{defGamma} we obtain $L_{x}(\Psi_{\alpha})=0$. Using \eqref{rotspeed} and \eqref{Expcase1}, the error function takes the form 
\begin{equation*}
\begin{aligned}
\ve^2 S\left(\Psi_{\alpha}\right)&=
\Bigg[\eta_{\delta}(\nu'(1)\log(|y|)-\nu'(1)|\log\ve|)\left(1+c_1\ve y_1+c_2\ve^2|y|^2\right)+\ve y_1 \frac{h}{\sqrt{h^2+r_0^2}}\partial_{y_1}H_{2\ve}(P)\\
&\hspace{4mm}+\mathcal{O}(\ve^2|y|^2)-\frac{\alpha \,r_0 h }{\sqrt{h^2+r_0^2}}\ve y_1 |\log\ve|-\frac{\alpha}{2}\ve^2|\log\ve|\left(\frac{h^2}{h^2+r_0^2}y_1^2+y_2^2\right)+\nu'(1)|\log\ve|\Bigg]^{\gamma}_{+}
\\
&=\left(\nu'(1)|\log\ve|\tilde{E}(y)\right)^{\gamma}_{+},
\end{aligned}
\end{equation*}
where 
\begin{align*}
\tilde{E}(y)&=\eta_{\delta}\left(\frac{\log(|y|)}{|\log\ve|}\left(1+c_1\ve y_1 + c_2\ve^2 |y|^2\right)-\left(1+c_1\ve y_1 +c_2\ve^2|y|^2\right)\right)+\mathcal{O}(\delta)\\
&+\mathcal{O}\left(\frac{\delta}{|\log\ve|}\right)+\mathcal{O}\left(\frac{1}{|\log\ve|}\right)+\frac{\ve^2}{4(h^2+r_0^2)}\left(\frac{h^2}{h^2+r_0^2}y_1^2+y_2^2\right) +1.
\end{align*}
Due to the expansion $\frac{\log(|y|)}{|\log\ve|}=1+\mathcal{O}\left(\frac{\log\delta}{|\log\ve|}\right)$, a careful analysis shows that for $\delta>0$ sufficiently small, there exists $\ve_{0}>0$ such that $\tilde{E}(y)> 0$ for all $\ve \in (0,\ve_{0})$. Since $\nu'(1)|\log\ve|<0$, it then follows immediately that 
\begin{equation}\label{errorinner2}
\ve^2 S\left(\Psi_{\alpha}\right)=0, \quad \text{for all} \quad 1<|y|<\frac{\delta}{\ve}.
\end{equation} 
\text{\bf Case 3:} $|y| \geq \frac{\delta}{\ve}.$
\\
In this unbounded region we have $\eta_{\delta}(x)=0$, thus using \eqref{boundh2e} we infer 
\begin{equation}\label{errorouter}
\ve^2 S\left(\Psi_{\alpha}\right)=\left(\frac{\alpha}{2}|\log\ve|r_0^2+H_{2\ve}(x)-\frac{\alpha}{2}|\log\ve||x|^2+\nu'(1)|\log\ve|\right)^{\gamma}_{+}=0,
\end{equation}
since for sufficiently small $\ve>0$ the dominant term is given by $-\frac{\alpha}{2}|\log\ve||x|^2.$

Collecting \eqref{errorinner1},\eqref{errorinner2} and \eqref{errorouter}, we globally write \[\ve^2 S(\Psi_{\alpha})\leq C\ve |y|\Gamma^{\gamma}_{+},\]
which concludes the proof.
\end{proof}

\section{Inner-Outer Gluing Scheme}\label{Section 4}
In this section, we consider the approximate stream function $\Psi_{\alpha}$ in \eqref{globalapprox} satisfying the error estimates of Proposition \ref{errorprop1}, and look for an exact solution $\Psi$ of the equation 
\begin{equation}\label{truesol}
  S\left(\Psi\right)(x)=L_{x}\left(\Psi\right)+f\left(\Psi-\frac{\alpha}{2}|\log\varepsilon||x|^2\right)=0 \quad \text{in} \quad \R^2, \quad f(s)=\frac{1}{\ve^2}\left(s+\nu'(1)|\log\ve|\right)^{\gamma}_{+},
\end{equation}
as a small perturbation around the approximate solution constructed. We recall that the rotational speed is fixed as $\alpha=\frac{-\nu'(1)}{2(h^2+r_0^2)} +\mathcal{O}\left(|\log\ve|^{-1}\right)$ in \eqref{rotspeed}. 

Using the \emph{Inner–Outer gluing scheme}, which is a perturbative method for concentration problems in nonlinear PDEs, we seek a solution as a small perturbation of $\Psi_{\alpha}$ of the form
\begin{equation}\label{perturbphi}
    \Psi=\Psi_{\alpha}+\varphi(x), \quad \varphi(x)=\tilde{\eta}_{\delta}(x)\phi_{in}\left(\frac{z}{\varepsilon}\right)+\phi_{out}(x),
\end{equation}
where 
\begin{equation}\label{etadelta}
\tilde{\eta}_{\delta}(x)=\eta\left(\frac{|\log\ve|^2 |z|}{\delta}\right),
\end{equation}
with $\eta$ defined in \eqref{definitioneta}. Notice that the cut-off $\tilde{\eta}_{\delta}(x)$ considered here is shorter than the one used in \eqref{globalapprox}, with this choice localising the contribution of the  perturbation $\phi_{in}$ in a ball of radius $\frac{\delta}{|\log\ve|^2}$ around the point of concentration.

We emphasise that the specific structure of the perturbation $\varphi(x)$ is fundamental to the gluing scheme. In particular, the perturbation splits into an inner function $\phi_{in}$ in the expanded variable $y=\frac{z}{\ve}$, designed to provide the local correction near the concentration point, as well as a function $\phi_{out}(x)$ in the original variable $x$, which accounts for the cancellation of the far-field error. 

By substituting \eqref{perturbphi} into \eqref{truesol} and linearising with respect to $\varphi$, the problem \eqref{truesol} reads 
\begin{equation*}
    S\left(\Psi_{\alpha} +\varphi\right)=\mathcal{E}_{\alpha}+\mathcal{L}_{\Psi_\alpha}[\varphi]+N_{\Psi_a}[\varphi]=0 \quad \text{in} \quad \R^2,
\end{equation*}
where 
\begin{equation*}
    \begin{aligned}
      &\mathcal{E}_{\alpha}=S(\Psi_{\alpha})
      ,\\
      &\mathcal{L}_{\Psi_{\alpha}}[\varphi]=L_x[\varphi]+f'\left(\Psi_{\alpha} -\frac{\alpha}{2}|\log\ve||x|^2\right)\varphi, 
      \\
      &N_{\Psi_{\alpha}}[\varphi]=f\left(\Psi_{\alpha} -\frac{\alpha}{2}|\log\ve||x|^2+\varphi\right)-f\left(\Psi_{\alpha} -\frac{\alpha}{2}|\log\ve||x|^2\right)-f'\left(\Psi_{\alpha} -\frac{\alpha}{2}|\log\ve||x|^2\right)\varphi. 
      \end{aligned}
\end{equation*}
The problem can be equivalently rearranged as
$$
	\begin{aligned} 
    S\left( \Psi_{\alpha} +\varphi\right) &= 0 \quad \text{in} \quad \R^2, \\
		S\left(\Psi_{\alpha}+\varphi\right)&=\
        		\tilde{\eta}_{\delta} \left[ L_x\left[\phi_{in}\right] +   f'\left(\Psi_{\alpha}-\frac{\alpha}{2}|\log\ve||x|^2\right)\left(\phi_{in} +\phi_{out}\right)  + S \left(\Psi_{\alpha} \right) +  N_{\Psi_\alpha} (\varphi)\right] \\
		& +\left(1-\tilde{\eta}_{\delta}\right)\left[f'\left( \Psi_{\alpha} - \frac{\alpha}{2}|\log\ve||x|^2\right) \phi_{out}  +   S(\Psi_{\alpha}) +  N_{\Psi_{\alpha} } (\varphi) \right]\\
		&+    L_x\left[\phi_{out}\right] + L_x\left[\tilde{\eta}_{\delta} \phi_{in}\right]- \tilde{\eta}_{\delta}  L_x[\phi_{in}].
	\end{aligned}
	$$
We then realise that $\Psi$ in \eqref{perturbphi} is an exact solution of \eqref{truesol} if and only if $(\phi_{in},\phi_{out})$ in the decomposition of $\varphi$ solve the coupled system of equations 
 	\begin{equation}\label{inprob1}
	\begin{aligned}
		L_{x}[\phi_{in}]  + f'\left( \Psi_{\alpha}-\frac{\alpha}{2} |\log \ve| |x|^2\right ) (\phi_{in} + \phi_{out})  + S (\Psi_{\alpha} )+  N_{\Psi_{\alpha}} \left(\varphi\right) \, =\, 0, \quad |z|< \frac{2\delta}{|\log\ve|^2},
	\end{aligned}
 \end{equation}
and
\begin{equation}\label{outprob1}
 \begin{aligned}
L_{x}\left[\phi_{out}\right]&+ \left(1-\tilde{\eta}_{\delta} \right)\left[f'\left( \Psi_{\alpha}-\frac{\alpha}{2} |\log \ve| |x|^2\right) \phi_{out} + S\left(\Psi_{\alpha} \right) + N_{\Psi_{\alpha} } (\varphi) \right]\\
&+L_{x}[\tilde{\eta}_{\delta}\phi_{in}]-\tilde{\eta}_{\delta}L_{x}[\phi_{in}] = 0, \inn \R^2.
\end{aligned} 
\end{equation}
\\
In this context, we refer to \eqref{inprob1} as the {\it Inner problem} and to \eqref{outprob1} as the {\it Outer problem}. 

For our analysis, it is relevant to recast \eqref{inprob1} in the expanded variable $y=\frac{z}{\ve}$, where $|y|<\frac{2\delta}{\ve|\log\ve|^2}.$ According to Proposition \ref{expansionofL}, for any smooth function $\bar{\phi}(z)=\phi(y)$ in this region, one can observe that $\ve^{2} L_{x}\left(\bar{\phi}(\ve y)\right)=\Delta_{y}\bar{\phi} +\tilde{B}_{0}(y)[\bar{\phi}]$, where $\tilde{B}_{0}=\ve^2 B_{0}\left(\ve y\right)$, see \eqref{oper}.

In addition, for $|y|<\frac{2\delta}{\ve|\log\ve|^2}$ we can expand \[\ve^2 f'\left(\Psi_{\alpha}-\frac{\alpha}{2}|\log\ve||x|^2\right)=\gamma\Gamma^{\gamma-1}_{+}+b_0(y),\] 
where 
\begin{equation}\label{defofbzero}
\begin{aligned}
b_0(y)=\gamma(\gamma-1)\Gamma^{\gamma-2}_{+}\Bigg[&\ve|\log\ve|y_1\left(-c_1\nu'(1)-\frac{\alpha r_0 h}{\sqrt{h^2+r_0^2}}\right)+\ve y_1 \left(c_1\Gamma(y)+\frac{h}{\sqrt{h^2+r_0^2}}\partial_{y_1}H_{2\ve}(P)\right)\\&+\mathcal{O}\left(\ve^2|\log\ve||y|^2\right)\Bigg] +\mathcal{O}\left(\ve^2|y|^2 \Gamma^{\gamma-3}_{+}\right),
\end{aligned}
\end{equation}
and $b_0(y_1,y_2)=b_0(y_1,-y_2).$

Arguing in a similar manner as above, the nonlinear quadratic term has the expansion 
\[\mathcal{N}_{\Psi_\alpha}(\varphi)=\ve^2 N_{\Psi_\alpha}(\tilde{\eta}_{\delta}\phi_{in}+\phi_{out})=\left(\gamma\Gamma^{\gamma-1}_{+}+b_0(y)\right)\mathcal{O}\left(\varphi^2\right),\]
hence if we multiply equation \eqref{inprob1} by $\ve^2$, the Inner problem can be expressed as
\begin{equation}\label{inprobexp}
\Delta_{y}\phi_{in}+\gamma\Gamma^{\gamma-1}_{+}\phi_{in}+\bar B[\phi_{in}]+\mathcal{N}_{\Psi_\alpha}\left(\tilde{\eta}_{\delta}\phi_{in}+\phi_{out}\right)+\ve^2 S(\Psi_{\alpha})+\left(\gamma\Gamma^{\gamma-1}_{+}+b_0(y)\right)\phi_{out}=0 \quad\mbox{in}\,\,B_{\rho},
\end{equation}
where $\rho=\frac{2\delta}{\ve|\log\ve|^2}$ and $\bar{B}[\phi_{in}](y)=\ve^2 B_0(\ve y)[\phi_{in}]+b_0(y)\phi_{in}(y)$, as in \eqref{oper} and \eqref{defofbzero}.
\medskip

As far as the Outer problem in \eqref{outprob1} is concerned, to simplify notation we write 
\begin{equation}\label{Poissonouter}
L_{x}\left[\phi_{out}\right]+G_{out}\left(\phi_{in},\phi_{out}\right)= 0 \quad\mbox{in} \quad\R^2,
\end{equation}
where  
\begin{equation*}
\begin{aligned}
G_{out}\left(\phi_{in},\phi_{out}\right)=&\left(1-\tilde{\eta}_{\delta}\right)\left[f'\left(\Psi_{\alpha}-\frac{\alpha}{2}|\log\ve||x|^2\right)\phi_{out}+S\left(\Psi_{\alpha}\right)+N_{\Psi_\alpha}\left(\tilde{\eta}_{\delta}\phi_{in}+\phi_{out}\right)\right]\\
&+L_{x}[\tilde{\eta}_{\delta}\phi_{in}]-\tilde{\eta}_{\delta}L_{x}[\phi_{in}].
\end{aligned}
\end{equation*}
\subsection{Solving the Outer problem}\label{subsecouter}
To address the coupled system of equations in \eqref{inprobexp} and \eqref{Poissonouter}, we begin by solving the Outer Problem in \eqref{Poissonouter}. For $\beta\in(0,1)$, we seek a solution 
\begin{equation}\label{b100}
\phi_{out}\in\mathcal{B}_{100}\coloneq\left\{f \in C^{1,\beta}(\R^2):\|(1+|x|^2)^{-1
}f\|_{\infty}\leq 100\right\},
\end{equation}
where $\mathcal{B}_{100}$ is a closed ball of the space $X\defeq \left\{f\in C^{1,\beta}(\R^2): \|(1+|x|^2)^{-1
}f\|_{\infty}< \infty\right\}.$

Regarding $\phi_{in}$, we assume a priori that in the rescaled variable $y=\frac{z}{\ve}$ it satisfies
\begin{equation}\label{propertiesforphi1}
(1+|y|)|D_y \phi_{in}(y)|+|\phi_{in}(y)| \leq \frac{C\ve}{1+|y|^{\sigma}}, \quad \phi_{in}(y_1,y_2)=\phi_{in}(y_1,-y_2),
\end{equation}
for some $\sigma>0$. These assumptions are motivated by the estimates of $\ve^2 S(\Psi_{\alpha})$ in Proposition \ref{errorprop1} and the fact that $\ve^2 S(\Psi_{\alpha})$ is even in $y_2$.

Under these assumptions, the following lemma holds.
\begin{lemma}\label{lemmaout}
Let $\phi_{out}$ as in \eqref{b100}
and $\phi_{in}$ satisfy \eqref{propertiesforphi1}. For  $\delta>0$ sufficiently small and the cut-off functions $\eta_{\delta}, \,\tilde{\eta}_{\delta}$ defined in \eqref{definitioneta} and \eqref{etadelta}, there exists a small $\ve_0 >0$ such that for all $\ve \in (0,\ve_0)$ it holds \[G_{out}\left(\phi_{in},\phi_{out}\right)=L_{x}\left[\tilde{\eta}_{\delta}\phi_{in}\right]-\tilde{\eta}_{\delta}L_{x}\left[\phi_{in}\right].\]
\end{lemma}
\begin{proof}
Since  $\left(1-\tilde{\eta}_{\delta}\right)\neq 0$ if and only if $ |z|>\frac{\delta}{2|\log\ve|^2}$, we employ the estimate $\ve^2 S(\Psi_\alpha) \leq C\ve |y| \Gamma^{\gamma}_{+}$ of Proposition \ref{errorprop1} to deduce  that \[(1-\tilde{\eta}_{\delta})S(\Psi_\alpha)=0,\quad (1-\tilde{\eta}_{\delta})f'\left(\Psi_\alpha - \frac{\alpha}{2}|\log\ve||x|^2\right)=0,\] since $\operatorname{supp} S\left(\Psi_{\alpha}\right)(z) \subset B_{\ve}(0)$.
It remains to examine the term $(1-\tilde{\eta}_{\delta})N_{\Psi_\alpha}\left(\tilde{\eta}_{\delta}\phi_{in}+\phi_{out}\right)$, where we distinguish between three separate regions in the $z$ variable, naturally suggested by the cut-offs $\tilde{\eta}_{\delta}$ and $\eta_{\delta}$. 
\\
\text{\bf Case 1}:  $\frac{\delta}{2|\log\ve|^2}< |z|\leq \frac{\delta}{|\log\ve|^2}$.
\\
In this region we have $\eta_{\delta}=1$ and $ \tilde{\eta}_{\delta}\neq 0$, thus using \eqref{rotspeed}, \eqref{Expcase1} and the variable $y=\frac{z}{\ve}$, we get
\begin{equation*}
\begin{aligned}
&(1-\tilde{\eta}_{\delta})N_{\Psi_\alpha}=\left(1-\tilde{\eta}_{\delta}\right)f\left(\Psi_{\alpha}-\frac{\alpha}{2}|\log\ve||x|^2+\tilde{\eta}_{\delta}\phi_{in}+\phi_{out}\right)\\
&=\frac{1-\tilde{\eta}_{\delta}}{\ve^2}\Bigg[\left(\nu'(1)\log(|y|)-\nu'(1)|\log\ve|\right)(1+c_1\ve y_1 +c_2\ve^2|y|^2) +\frac{h}{\sqrt{h^2+r_0^2}}\ve y_1\partial_{y_1}H_{2\ve}(P)+\mathcal{O}(\ve^2|y|^2)\\
&\hspace{15mm}+\frac{\nu'(1) r_0 h }{2(h^2+r_0^2)^{\frac 3 2}}\ve y_1 |\log\ve|+\mathcal{O}(\ve |y|)+\frac{\nu'(1)\ve^2|\log\ve|}{4(h^2+r_0^2)}\left(\frac{h^2}{h^2+r_0^2}y_1^2+y_2^2\right)+\tilde{\eta}_{\delta}\phi_{in}(y)+\phi_{out}(x)\Bigg]^{\gamma}_{+}.
\end{aligned}
\end{equation*}
Using \eqref{b100} and \eqref{propertiesforphi1}, we then have
\begin{equation*}
\begin{aligned}
    (1-\tilde{\eta}_{\delta})N_{\Psi_{\alpha}}=\frac{1}{\ve^2}&\Bigg[\nu'(1)|\log\ve|\Bigg(\left(\frac{\log(|y|)}{|\log\ve|}-1\right)(1+c_1\ve y_1+c_2\ve^2|y|^2)\\
    &+\mathcal{O}\left(\frac{\delta}{|\log\ve|}\right)
    +\mathcal{O}(\delta)+\mathcal{O}(\ve^{1+\sigma}|\log\ve|^{2\sigma-1})+\mathcal{O}\left(\frac{1+|x|^2}{|\log\ve|}\right)\Bigg)\Bigg]^{\gamma}_{+}.
\end{aligned}
\end{equation*}
Similarly to the proof of Proposition \ref{errorprop1}, choosing $\delta>0$ sufficiently small yields $(1-\tilde{\eta}_{\delta})N_{\Psi_\alpha}=0$ for all $\ve>0$ small. 
\\
\text{\bf Case 2}: $\frac{\delta}{|\log\ve|^2} < |z| <\delta$. 
\\
In this intermediate region it holds $\eta_\delta \neq 0$ and $\tilde{\eta}_{\delta}=0$, so we instead find
\begin{equation}\nonumber
\begin{aligned}
    &(1-\tilde{\eta}_{\delta})N_{\Psi_\alpha}=f\left(\Psi_{\alpha}-\frac{\alpha}{2}|\log\ve||x|^2+\phi_{out}\right)\\
    &=\frac{1}{\ve^2}\Bigg[\nu'(1)|\log\ve|\Bigg(\frac{r_0^2}{2(h^2+r_0^2)}+\eta_{\delta}\left(1+c_1 \ve y_1 +c_2 \ve^2 |y|^2\right)\left(\frac{\log(|y|)}{|\log\ve|}-1\right) \\
    &\hspace{9mm}+\mathcal{O}\left(\frac{\delta}{|\log\ve|}\right)+\mathcal{O}(\delta)+\frac{\ve^2}{4(h^2+r_0^2)}\left(\frac{h^2}{h^2+r_0^2}y_1^2+y_2^2\right)+\mathcal{O}\left(\frac{1+|x|^2}{|\log\ve|}\right)+1\Bigg)\Bigg]^{\gamma}_{+}.
\end{aligned}
\end{equation}
 Using an analogous argument as before, we get that for sufficiently small $\delta>0$ it holds $(1-\tilde{\eta}_{\delta})N_{\Psi_\alpha}=0$ for all $\ve>0$ small.
\\
\text{\bf Case 3:} $|z| \geq \delta.$
\\
In this unbounded region we have $\eta_{\delta}=0$ and $\tilde{\eta}_{\delta}=0,$
thus using \eqref{b100} and \eqref{boundh2e} yields
\[\left(1-\tilde{\eta}_{\delta}\right)N_{\Psi_{\alpha}}=\frac{1}{\ve^2}\left(\frac{\alpha}{2}|\log\ve|r_0^2+H_{2\ve}(x)-\frac{\alpha}{2}|\log\ve||x|^2+\phi_{out}(x)\right)^{\gamma}_{+}=0,\]
since for sufficiently small $\ve>0$ the dominant term is $-\frac{\alpha}{2}|\log\ve||x|^2<0$.
\end{proof}
\medskip

A direct Corollary of Lemma \ref{lemmaout} is that if $\phi_{out}$ and $\phi_{in}$ satisfy \eqref{b100} and \eqref{propertiesforphi1} respectively, then the Outer problem \eqref{Poissonouter} reduces to 
\begin{equation}\label{poisreduced}
L_{x}[\phi_{out}]+L_{x}[\tilde{\eta}_{\delta}\phi_{in}]-\tilde{\eta}_{\delta}L_{x}[\phi_{in}]=0 \quad \mbox{in} \quad \R^2.
\end{equation}
In other words, choosing a priori the appropriate topologies for $\phi_{in}$ and $\phi_{out}$ leads to a significant simplification, as the structure of the Poisson equation \eqref{poisreduced} allows to use directly Proposition \ref{propouter} to establish the existence of a solution, without invoking a fixed point scheme. In particular, since the term $L_{x}[\tilde{\eta}_{\delta}\phi_{in}]-\tilde{\eta}_{\delta}L_{x}[\phi_{in}]$ is independent of $\phi_{out}$, the previous conclusion holds provided that this term has sufficiently fast decay at infinity.

To verify this, a substitution in \eqref{oper} gives 
\begin{equation}\label{outeroperatorphiin}
    \begin{aligned}
        L_{x}\left[\tilde{\eta}_{\delta}\phi_{in}\right]-\tilde{\eta}_{\delta} L_{x}\left[\phi_{in}\right]&=
\nabla_{z}\eta\left(\frac{|\log\ve|^2|z|}{\delta}\right)\cdot\nabla_{z}\phi_{in}\left(\frac{z}{\ve}\right)+\eta\left(\frac{|\log\ve|^2|z|}{\delta}\right)\Delta_{z}\phi_{in}\left(\frac{z}{\ve}\right)\\
 &+\phi_{in}\left(\frac{z}{\ve}\right)\Delta_{z}\eta\left(\frac{|\log\ve|^2|z|}{\delta}\right)+B_{0}\left[\eta\left(\frac{|\log\ve|^2|z|}{\delta}\right)\phi_{in}\left(\frac{z}{\ve}\right)\right]\\
        &-\eta\left(\frac{|\log\ve|^2|z|}{\delta}\right)B_0\left[\phi_{in}\left(\frac{z}{\ve}\right)\right].
    \end{aligned}
\end{equation}
One can deduce that $L_{x}[\tilde{\eta}_{\delta}\phi_{in}]-\tilde{\eta}_{\delta}L_{x}[\phi_{in}]$
is compactly supported, hence there exists a constant  $\bar\nu >2$ such that \[\big|L_{x}[\tilde{\eta}_{\delta}\phi_{in}]-\tilde{\eta}_{\delta}L_{x}[\phi_{in}]\big| < \frac{C}{1+|x|^{\bar\nu}}.\] Furthermore, for $\sigma>0$ as in \eqref{propertiesforphi1}, direct computations yield \[\sup\limits_{x\in\R^2}\big|L_{x}\left[\tilde{\eta}_{\delta}\phi_{in}\right]-\tilde{\eta}_{\delta} L_{x}\left[\phi_{in}\right]\big| <C\ve^{1+\sigma^{*}},\]
for some arbitrarily small $0<\sigma^{*}<\sigma$. For instance, using \eqref{propertiesforphi1} and the variable $y=\frac z \ve$, we can estimate
\begin{equation}\nonumber
\begin{aligned}
\Bigg|\nabla_{z}\eta\left(\frac{|\log\ve|^2|z|}{\delta}\right)\cdot\nabla_{z}\phi_{in}\left(\frac{z}{\ve}\right)\Bigg|\leq \frac{C|\log\ve|^2
}{\ve}\eta'\left(\frac{|\log\ve|^2|z|}{\delta}\right)\nabla_{y}\phi_{in}\leq o(1) \frac{\ve ^{1+\sigma^{*}}}{1+|x|^{\bar{\nu}}},
\end{aligned}
\end{equation}
for $o(1)\to 0$ as $\ve \to 0,\, \bar{\nu}>2$ and $0<\sigma^{*}<\sigma$ arbitrarily small, with similar estimates for the remaining terms in \eqref{outeroperatorphiin}.

As a result, Proposition \ref{propouter} establishes the existence of a solution $\phi_{out}\in C^{1,\beta}(\R^2)$ for any $0<\beta<1$ to the Poisson equation \eqref{poisreduced}, with the bound 
\begin{equation}\label{boundphiout}
|\phi_{out}|\leq C \ve^{1+\sigma^{*}}\left(1+|x|^2\right).
\end{equation}
Finally, since the solution is given up to a constant, we choose 
\begin{equation}\label{phioutP}
\phi_{out}(r_0,0)=0.
\end{equation}
As it will be made evident in Proposition \ref{propsec7}, the property \eqref{phioutP} for $\phi_{out}$ is crucial for establishing existence and uniqueness of a solution of the Inner–Outer system in \eqref{inprob1} and \eqref{poisreduced}, as it yields sufficient decoupling of the equations and allows to employ a fixed point argument.
\section{Inner and Outer Linear Theories}\label{Section 5}
\subsection{Linear Outer Theory}
Recalling the elliptic differential operator in divergence form $L_x$ defined in \eqref{operL}, we consider the Poisson equation  
	\begin{equation}
	\label{outerprob}
	L_x[\phi_{out}]  +  g(x) = 0  \inn \R^2 ,
	\end{equation}
	where $g$ is a bounded function satisfying the decay condition
	$$
    \|g\|_{\bar{\nu}}\, :=\, \sup_{x\in \R^2}\left(1+|x|^{\bar{\nu}}\right)|g(x)|< +\infty ,
    $$
	for some $\bar{\nu} >2$. 
    
 The following Proposition holds.
\begin{prop}\label{propouter}
		There exists a solution $\phi_{out}(x)$ to \eqref{outerprob} belonging to the class $C^{1,\beta}(\R^2)$ for any $0<\beta<1$,  which defines a linear operator $\phi_{out}= {\mathcal T}^o (g) $ of $g$
		and satisfies the bound
		\begin{equation*}
		|\phi_{out}(x)| \,\le \,C\| g\|_{\bar{\nu}} (1+ |x|^2),
		\end{equation*}
		for a constant $C>0$.
	\end{prop}
\begin{proof}
See (\cite{MR4417384}, Proposition 7.1).
\end{proof}

\subsection{Linear Inner Theory}
In this section, for $y=\frac{z}{\ve}$ and $|h(y)|\leq \frac{C}{1+|y|^{2+\sigma}}$ for some $\sigma>0$, we study the linearised problem
\be\label{Linearisedinner}
\begin{aligned}
\Delta_{y} \phi_{in}(y) + \gamma\Gamma^{\gamma-1}_{+}\phi_{in}(y) +  h(y)= 0 \inn \R^2, 
\end{aligned}
\ee
where $\Gamma$ is given in \eqref{defGamma}.

To solve this equation, we work in a topology of decaying functions endowed with Hölder-type norms, that capture both the decay at infinity and the Hölder regularity of the functions involved. 

Following this direction, for $\sigma>0$ and $\beta\in(0,1)$ it is relevant to define the norms
\be\label{norma} \begin{aligned}
\|h\|_{2+\sigma}  = & \sup_{y\in \R^2} (1+|y|)^{2+\sigma}|h(y)|, \\
\|h\|_{2+\sigma,\beta}  = &\|h\|_{2+\sigma}  +   \sup_{y\in \R^2}(1+|y|)^{2+\sigma + \beta}[h]_{B_1(y),\beta} ,
\end{aligned}
\ee
where for any $A\subset \R^2$ we use the standard semi-norm notation
$$
[h]_{A,\beta} = \sup_{\ell_1\neq \ell_2, \ell_1,\ell_2\in A}  \frac {|h(\ell_1) -h(\ell_2)| } {|\ell_1-\ell_2|^\beta}.
$$

A well-known result of Dancer and Yan \cite{MR2456896} establishes that the linear operator $\Delta_y +\gamma\Gamma^{\gamma-1}_{+}$ in \eqref{Linearisedinner} is non-degenerate, in the sense that its kernel in $L^{\infty}(\R^2)$ is spanned by the functions $Z_i(y)=\frac{\partial\Gamma}{\partial y_{i}}, \ i=1,2,$ which is consistent with \eqref{limit} being invariant under translations. Nevertheless, in the absence of suitable orthogonality conditions, the solution $\phi_{in}$ of \eqref{Linearisedinner} grows logarithmically fast at infinity, so the previous non-degeneracy result fails in our setting. In particular, owing to the logarithmic growth, the kernel of the linear operator $\Delta_y +\gamma\Gamma^{\gamma-1}_{+}$ in $\R^2$ is now spanned by
\begin{equation}\label{elemkernel}
Z_0=\frac{2}{\gamma-1}\Gamma(y)+y\cdot\nabla_y \Gamma(y), \quad  Z_i(y)  =  \partial_{y_i}  \Gamma(y), \  i=1,2.
\end{equation}
To identify the radial element $Z_0$, we observe that for any $\lambda,\zeta>0$, the equation \eqref{rescaledequation123} is invariant under the scaling $y \mapsto \lambda y, \Gamma \mapsto \lambda^{\zeta}\Gamma.$ From this perspective, we define the rescaled profile
$\Gamma_{\lambda}(y)=\lambda^{\zeta}\Gamma(\lambda y)$, where balancing the coefficients of the linear and nonlinear terms gives $\zeta=\frac{2}{\gamma-1}.$ Differentiating at \(\lambda=1\), one then obtains the generator of the scaling 
\[Z_0=\partial_{\lambda}\left(\lambda^{\frac{2}{\gamma-1}}\Gamma(\lambda y)\right)\Big |_{\lambda=1}=\frac{2}{\gamma-1}\Gamma(y)+y\cdot\nabla_{y}\Gamma(y).\]
We remark that for $i=1,2,$ it holds 
$Z_i=\mathcal{O}\left(\frac{1}{1+|y|}\right)$ as $|y|\to\infty$, whereas 
$Z_0=\mathcal{O}\left(\log\left(2+|y|\right)\right)$ as $|y|\to\infty$. To the best of our knowledge, this is the first work in which the element $Z_0$ in the kernel of the linear operator $\Delta_y+\Gamma^{\gamma-1}_{+}$ in $\R^2$ cannot be discarded from the analysis, as related works were restricted to bounded solutions.

We obtain the following lemma.
\begin{lemma}\label{lemat}
Given $\sigma>0$ and $\beta\in (0,1)$, consider the norms defined in \eqref{norma}. Then,
there exist a constant $C>0$ and a solution $\phi_{in} =  \mathcal{T} [ h]$ of equation \eqref{Linearisedinner} for each $h$ with $\|h\|_{2+\sigma} <+\infty$, that defines a linear operator of $h$ and satisfies the estimate
\begin{equation*}
\begin{aligned}
&(1+|y|) |\nabla \phi_{in} (y)|  + | \phi_{in} (y)| \\
&\le  \,  C \left [ \,  \log \left(2+|y|\right) \left|\int_{\R^2} h Z_0\right|  +    (1+|y|) \sum_{i=1}^2 \left|\int_{\R^2} h Z_i\right|+  (1+|y|)^{-\sigma} \|h\|_{2+\sigma}   \,\right].
\end{aligned}
\end{equation*}
In addition, if 
$\|h\|_{2+\sigma,\beta} <+\infty$, it holds
\begin{equation*}
\begin{aligned}
& (1+|y|^{2+\beta})  [D^2_y \phi_{in}]_{B_1(y),\beta}  +\left(1+|y|^2\right)  \left|D^2_y \phi_{in} (y)\right| \\  &
\,  \le  \,   C \left [ \,  \log \left(2+|y|\right) \left|\int_{\R^2} h Z_0\right|  +    (1+|y|) \sum_{i=1}^2 \left|\int_{\R^2} h Z_i\right|  +  (1+|y|)^{-\sigma} \|h\|_{2+\sigma,\beta}   \,\right]. 
\end{aligned}
\end{equation*}
\end{lemma}

\proof
Assuming that $h$ is a complex-valued function, we turn to polar coordinates $y=re^{i\theta}, \, r=|y|$,  and decompose $\phi_{in}$ and $h$ in Fourier series as 
$$ h(y) =  \sum_{k=-\infty}^{\infty}  h_k (r) e^{ik\theta} , \quad \phi_{in}(y)= \sum_{k=-\infty}^{\infty} \phi_{in,k} (r) e^{ik\theta}.
$$
It is then immediate that equation \eqref{Linearisedinner} decouples into modes, yielding the infinitely many ordinary differential equations  
\begin{equation}\label{01}
L_k [\phi_{in,k}] + h_k(r) = 0 , \quad r\in(0,\infty), \quad k \in \Z,
\end{equation}
where
$$
L_k\defeq \partial_{rr} + \frac 1r \partial_{r}- \frac{k^2}{r^2}+  \gamma\Gamma^{\gamma-1}_{+}  .
$$
For $k=0$, we have that the function $z_0(r) = \frac{2}{ \gamma-1}\Gamma(r)+r\Gamma'(r)$ satisfies $L_0[z_0]=0,$ thus if $\xi_0\in(0,1)$ denotes the unique root of $z_{0}$, the function defined as
$$
\phi_{in,0} (r) =     -z_0(r)\int_{\xi_0}^r  \frac {\rm d s}{ s z_0(s)^2}  \int_0^s h_0(\rho) z_0(\rho) \rho\, \dd \rho
$$
is a smooth solution of \eqref{01}.

Since $\int_0^\infty h_0(\rho) z_0(\rho) \rho\, \dd\rho = \frac 1{2\pi}\int_{\R^2} h(y)Z_0(y)\, \dd y,      $
we deduce that 
$$
|\phi_{in,0}(r)| \, \le \,   C\Big[ \, \log (2 + r) \Big| \int_{\R^2} h(y)Z_0(y)\, \dd y      \Big|   \, +\,  (1+r)^{-\sigma} \|h\|_{2+ \sigma} \Big ].
$$
For $|k|=1$, we observe that for $z_k(r)=-\Gamma'(r)$ it holds $L_k[z_k]=0$, so the function given by
$$
\phi_{in,k}  (r) =     \Gamma'(r)\int_r^\infty  \frac {\dd s}{ s\Gamma'(s)^2}  \int_0^s h_k(\rho) \Gamma'(\rho) \rho\, \dd\rho
$$
solves \eqref{01} for $k =-1,1,$ and satisfies
$$
|\phi_{in,k}(r)| \, \le \,   C\Big[ \,  (1+ r) \sum_{j=1}^2 \Big| \int_{\R^2} h(y)Z_j(y)\, \dd y      \Big|   \, +\,  (1+r)^{-(1+\sigma)} \|h\|_{ 2+\sigma} \Big ].
$$
For $k=2$, standard asymptotics give the existence of a function $z_2(r)$ such that $ L_2[z_2] = 0 $ and  $z_2(r)=\mathcal{O}(r^{2})$ as $r\to 0$ and $r\to \infty$, while for $|k|\ge 2$ we have that the function
$$
\bar\phi_{in,k}  (r) =   \frac 4{ k^2} z_2(r)\int_0^r  \frac {\dd s}{ s z_2(s)^2}  \int_0^s |h_k(\rho)| z_2(\rho) \rho\,\dd \rho
$$
is a positive supersolution for equation \eqref{01}. As such, we get the existence of a unique solution $\phi_{in,k}$  with $|\phi_{in,k}(r)| \le \bar\phi_{in,k}  (r)$, so that
$$
|\phi_{in,k}(r)| \, \le \,    \frac C {k^2}   (1+r)^{-\sigma} \|h\|_{2+\sigma}, \quad |k|\ge 2.
$$
Combining the above, it is clear that the function  $$\phi_{in}(y)= \sum_{k=-\infty}^{\infty} \phi_{in,k} (r) e^{ik\theta} $$  solves \eqref{Linearisedinner} and defines a linear operator of $h$ , while adding the individual estimates for each mode we infer that
\begin{equation}\label{cot} \begin{aligned}
 | \phi_{in} (y)|
\,  \le  \,  C \left [ \,  \log (2+|y|) \,\left|\int_{\R^2} h Z_0\right|  +    (1+|y|) \sum_{i=1}^2 \left|\int_{\R^2} h Z_i\right|  +  (1+|y|)^{-\sigma}\|h\|_{2+\sigma}   \,\right ]. \end{aligned}
\end{equation}
To proceed, assuming that $\|h\|_{2+\sigma,\beta} <+\infty$, we now turn our attention to proving similar estimates for the first-order and second-order derivatives of $\phi_{in}$. 
To do so, we let $R=|y|\gg 1$, fix a direction $\vec e=\frac{y}{|y|}\in\mathbb{S}^1$ so that $y=R\vec e$, and define the rescaled functions
\[
\phi_{R}(\ell)  = {R^{\sigma}} \phi_{in} \left(y+R\ell\right), \quad h_R(\ell) = R^{2+\sigma}h \left(y+R\ell\right). \]
In the neighbourhood $U_{y}\defeq \{y+R\ell: |\ell|<\frac{1}{2}\}$, a direct calculation shows that 
\[\Delta_\ell\phi_{R}(\ell)  +   R^2 \gamma\Gamma^{\gamma-1}_{+}(y+R\ell)\phi_{R}(\ell) + h_R (\ell)  = 0.\]
Setting \[\mu_i \defeq  \Big| \int_{\R^2}  h Z_i \Big|, \quad i=0,1,2,\]
we employ \eqref{cot} and a standard elliptic estimate to obtain
 \begin{equation}\label{holder1}
 \|\nabla_{\ell} \phi_{R} \|_{L^{\infty}\big( B_{\frac 14}(0) \big)} + \|\phi_{R} \|_{L^{\infty}\big( B_{\frac 12}(0) \big)} \, \le\,  C\Big [ \mu_0 R^{\sigma}\log R + \sum_{i=1}^2\mu_i R^{1+\sigma} + \| h\|_{2+\sigma} \Big ],
 \end{equation}
where we used that $\|h_{R}\|_{L^{\infty}\big(B_{\frac 1 2}(0)\big)}\leq C \|h\|_{2+\sigma}.$

In addition, since it also holds $[h_R(\ell) ]_{B_{\frac 12} (0),\beta}\leq C\| h\|_{2+\sigma,\beta}$ , an application of interior Schauder estimates together with the bounds in \eqref{holder1} yields
\begin{equation}\label{holder2}
\|D^2_{\ell} \phi_{R} \|_{L^\infty\big( B_{\frac 14}(0) \big) } +  [ D^2_{\ell} \phi_{R} ]_{B_{\frac 1 4}(0),\beta} \, \le\,  C\Big [ \mu_0 R^{\sigma} \log R +   \sum_{i=1}^2\mu_i R^{1+\sigma} +   \| h\|_{2+\sigma,\beta} \Big ].
\end{equation}
Collecting together \eqref{cot}, \eqref{holder1} and \eqref{holder2} we get the desired bounds, hence the proof is completed.
\qed
\section{Projected Linearised Inner Problem}\label{Section 6}
In the current section, we fix a sufficiently small $\delta>0$, and for $\rho= \frac{2\delta}{\ve|\log\ve|^2}$ and scalars $d_j, j=0,1,2,$ we focus on the projected linearised equation 
\begin{equation}\label{problemlinearprojected}
    \Delta_{y}\phi_{in}+\gamma\Gamma^{\gamma-1}_{+}\phi_{in}+\bar{B}[\phi_{in}]+h(y)=\sum_{j=0}^{2}d_j\gamma\Gamma^{\gamma-1}_{+}Z_j\quad\mbox{in}\quad B_{\rho}.
\end{equation}
 For convenience, we recall that the functions $Z_j, j=1,2$ are defined in \eqref{elemkernel}, while 
 \begin{equation}\label{defofBinner}
 \bar{B}[\phi_{in}]=b_0(y)\phi_{in}+\tilde{B}_0[\phi_{in}],
 \end{equation}
 with $b_0(y)=\mathcal{O}\left(\ve |y| \Gamma^{\gamma-2}_{+}\right)$ and $\tilde{B}_{0}=\ve^2B_{0}(\ve y)$, see \eqref{oper} and \eqref{defofbzero}. 

For $\sigma>0$ and  $\beta\in(0,1)$, we also adapt the norms in \eqref{norma} to be defined on any open set $A \subset \R^2$, by introducing 
\begin{equation}\label{norm1}
\|h\|_{2+\sigma,A}=\sup_{y\in A}\left(1+|y|\right)^{2+\sigma}|h(y)|, \quad \|h\|_{2+\sigma,\beta,A}=\sup_{y\in A}\left(1+|y|\right)^{2+\sigma+\beta}[h]_{B_{1}(y)\cap A}+\|h\|_{2+\sigma,A},\end{equation}
while for $\phi_{in}\in C^{2,\beta}\left(\R^2\right)$ we write
 \begin{equation}\label{starnorm}
 \|\phi_{in}\|_{*,\sigma,A} =  \|  D^2_y\phi_{in} \|_{2+\sigma,\beta,A}   + \|  D_y\phi_{in} \|_{1+\sigma,A}+ \|\phi_{in} \|_{\sigma,A}.
 \end{equation}
In what follows, for the sake of simplicity we omit the subscript when $A=\R^2$.

 We prove the following Proposition.
\begin{prop}\label{prop1} For any $\sigma>0$ and $\beta \in (0,1)$, consider the norms defined in \eqref{norm1} and \eqref{starnorm} respectively. Then, for all $\ve>0$ sufficiently small, $\rho=\frac{2\delta}{\ve|\log\ve|^2}$ and $h$ such that $\|h\|_{2+\sigma,\beta,B_{\rho}} < +\infty$, there exist numbers $\delta>0$ and $C>0$ such that for any differential operator $\bar{B}$ as in \eqref{defofBinner},
equation \eqref{problemlinearprojected} has a solution $\phi_{in} = T[h]$ for certain scalars $d_i= d_i[h], \,i=0,1,2$, that defines a linear operator of $h$ and satisfies
\begin{align*} 
\| \phi_{in}\|_{*,\sigma, B_\rho} \ \le\ C \|h \|_{2+\sigma,\beta, B_\rho} .
\end{align*}
Moreover, the functionals $d_i$ have the form
\begin{align*}
d_0[h]\, = & \, \gamma_0\int_{B_{\rho}}  h Z_0  + \mathcal{O}\left(\frac{\log(2+\rho)}{1+\rho^{\sigma}}\right)\|h \|_{2+\sigma,\beta, B_\rho} , \\    d_i[h]\, = &\, \gamma_i\int_{B_{\rho}}  h Z_i  + \mathcal{O}\left(\frac{1}{1+\rho^{1+\sigma}}\right)  \|h \|_{2+\sigma,\beta, B_\rho}, \ i=1,2,
\end{align*}
where
$\gamma_i^{-1} = \int_{\R^2} \gamma\Gamma^{\gamma-1}_{+} Z_i^2$, $i=0,1,2.$ \end{prop}

\proof
We begin by considering a standard linear extension operator $h\mapsto \tilde{h} $ to the whole $\R^2$,
such that the support of $\tilde{h}$ is contained in $B_{2\rho}$ and $\|\tilde{h}\|_{\sigma,\beta} \le C\|h\|_{\sigma,\beta, B_\rho}$, where $C>0$ is independent of $\rho$. Similarly,
we assume that the coefficients of $\bar{B}$ are of class $C^1(\R^2)$, with compact support in $B_{2\rho}$.
 With these in mind, we consider the auxiliary projected problem in the full space $\R^2$ given by

\begin{equation}\label{001} 
\Delta_y\phi_{in}  +  \gamma\Gamma^{\gamma-1}_{+}\phi_{in}  + \bar{B}[\phi_{in}]  + \tilde{h}(y)  = \sum_{j=0}^2  d_j\gamma\Gamma^{\gamma-1}_{+} Z_j \inn \R^2,
\end{equation}
where $d_i = d_i[h,\phi_{in}]$ are the scalars defined as
\begin{equation}\label{coeffsforortho}
d_i= \gamma_i \int_{\R^2} (\bar{B}[\phi_{in}]  + \tilde{h}(y))Z_i,  \quad \gamma_i^{-1} = \int_{\R^2}  \gamma\Gamma^{\gamma-1}_{+} Z_i^2, \quad i=1,2.
\end{equation}
On the one hand, for $i=1,2$ we estimate 
\[\bar{B}[Z_i]\leq C\ve\left(|y||D^2_{y}Z_i|+|D_y Z_i|\right) +\mathcal{O}\left(\ve |y|\Gamma^{\gamma-2}_{+}\right)Z_i  \leq \frac{C\ve}{1+|y|^2},\]
thus using that $ \int_{\R^2} \frac {1}{1+|y|^{2+\sigma}}  <+\infty$ we get
 $$
  \int_{\R^2} \bar{B}[\phi_{in}]Z_i =   \int_{\R^2} \phi_{in} \bar{B}[Z_i] \leq C\ve\|\phi_{in}\|_{*,\sigma}.
 $$
On the other hand, the case $i=0$ needs to be treated separately, since $Z_0(y)=\mathcal{O}\left(\log\left(2+|y|\right)\right)$ as $|y|\to\infty.$ Specifically, we find 
 \[\bar{B}[Z_0]\leq \frac{C\ve}{1+|y|},\]
 hence
 \begin{equation*}
 \begin{aligned}
\int_{\R^2} \bar{B}[\phi_{in}]Z_0 &\leq C\ve\int_{B_{2\rho}} \left(|y||D^2_{y}\phi_{in}|+|D_{y}\phi_{in}|\right)Z_0 +\int_{B_{2\rho}} \mathcal{O}\left(\ve |y|\Gamma^{\gamma-2}_{+}\right)Z_0\\
&\leq C\ve^{\sigma}|\log\ve|^{2\sigma-1}\|\phi_{in}\|_{*,\sigma}.
\end{aligned}
\end{equation*}
Moreover, one can also derive the estimates 
\[\int_{\R^2\setminus B_{2\rho}} \tilde{h}Z_0=\mathcal{O}\left(\frac{\log(2+\rho)}{1+\rho^{\sigma}}\right)\|h\|_{2+\sigma,\beta,B_{\rho}}, \quad \int_{\R^2\setminus B_{2\rho}} \tilde{h}Z_i=\mathcal{O}\left(\frac{1}{1+\rho^{1+\sigma}}\right)\|h\|_{2+\sigma,\beta,B_{\rho}}, \, i=1,2,\]
and 
\[\bar{B}[\phi_{in}] \leq\frac{C\ve}{1+|y|^{1+\sigma}}\|\phi_{in}\|_{*,\sigma},\]
which further yields \[\big\| \bar{B}[\phi_{in}] \big\|_{2+\sigma,\beta} \leq \frac{C}{|\log\ve|^2}\|\phi_{in}\|_{*,\sigma,B_{\rho}}.\]
We consider the Banach space $\mathcal{X}\defeq\{\phi_{in} \in C^{2,\beta}(\R^2)\ : 
  \|\phi_{in}\|_{*, \sigma} <+\infty\}$ and observe that in order to find a solution of \eqref{001},
it suffices to solve the equation
\be\label{fp}
 \phi_{in}  =   \mathcal A  [\phi_{in}]  +  \mathcal {H},\quad \phi_{in} \in \mathcal{X},
\ee
where
$$
\mathcal A  [\phi_{in}]
 =  \mathcal T\Big [\bar{B}[\phi_{in}] -\sum_{i=0}^2 d_i[0, \phi_{in}]\gamma\Gamma^{\gamma-1}_{+}Z_i  \Big] ,\quad
 \mathcal H  = \mathcal T\Big [  \tilde h   - \sum_{i=0}^2 d_i[\tilde h,0] \gamma\Gamma^{\gamma-1}_{+}Z_i  \Big],
$$
and $\mathcal T$ is the linear operator constructed in Lemma \ref{lemat}.

Due to \eqref{coeffsforortho} and the symmetries of $Z_j,\,j=0,1,2$, it follows that
\[\big\|\mathcal{A}[\phi_{in}]\big\|_{*,\sigma}\leq C\frac{\delta}{|\log\ve|^2}\|\phi_{in}\|_{*,\sigma},\]
and likewise 
\[\|\mathcal{H}\|_{*,\sigma}\leq C\|h\|_{2+\sigma,\beta,B_{\rho}}.\]
As a consequence, using the Contraction Mapping Theorem in $\mathcal{X}$ we deduce that the fixed point problem \eqref{fp} admits a unique solution for all sufficiently small $\ve>0$, which defines a linear operator of $h$ and satisfies
$$
\|\phi_{in} \|_{*, \sigma} \ \le\  C \|h \|_{ 2+\sigma,\beta, B_{\rho}}.
$$
To this end, we conclude the proof by setting $T[h] = \phi_{in}\big|_{B_\rho}$. \qed

\section{Resolution of the Inner-Outer Gluing System}\label{Section projected}
Having established the existence of a solution $\phi_{out}$ for the Outer problem \eqref{poisreduced} in Section \ref{subsecouter}, our current objective is to employ the linear theories developed in Sections \ref{Section 5} and \ref{Section 6} to find a solution $\phi_{in}$ for the Inner problem
\[\Delta_{y}\phi_{in}+\gamma\Gamma^{\gamma-1}_{+}\phi_{in}+\bar{B}[\phi_{in}]
+\mathcal{N}\left(\tilde{\eta}_{\delta}\phi_{in}+\phi_{out}\right)
+\ve^2 S(\Psi_\alpha)+\left(\gamma\Gamma^{\gamma-1}_{+}+b_0(y)\right)\phi_{out} = 0, \quad \mbox{in} \,\, B_{\rho}.\]
For clarity, we recall that $\rho=\frac{2\delta}{\ve|\log\ve|^2}$,  $\tilde{\eta}_{\delta}=\eta\left(\frac{|\log\ve|^2|z|}{\delta}\right)$ with $\eta$ as in \eqref{definitioneta}, and $\bar{B}$ was defined in \eqref{defofBinner}.

As it will become clear later on, it is convenient to decompose $\phi_{in}=\phi_1+\phi_2$ as follows. For the functions $Z_j$ given in \eqref{elemkernel} and scalars $d_j, j=0,1,2$ as in Proposition \ref{prop1},  we ask that $\phi_{1}$ solves the auxiliary projected problem
\begin{equation}\label{problemforphi1}
\Delta_{y}\phi_1+\gamma\Gamma^{\gamma-1}_{+}\phi_1+\bar{B}[\phi_1]+\bar{B}[\phi_2]+H(\phi_1+\phi_2,\phi_{out})=\sum_{j=0}^{2}d_j\gamma\Gamma^{\gamma-1}_{+}Z_j \quad \mbox{in} \quad B_{\rho},
\end{equation}
with\[H(\phi_1+\phi_2,\phi_{out})\defeq\mathcal{N}\left(\tilde{\eta}_{\delta}\phi_{in}+\phi_{out}\right) +\ve^2 S(\Psi_\alpha) +\left(\gamma\Gamma^{\gamma-1}_{+}+b_0(y)\right)\phi_{out}.\]
Regarding $\phi_2$, we require that it satisfies the linear equation 
\begin{equation}\label{problemforphi2}
\Delta_{y}\phi_2+\gamma\Gamma^{\gamma-1}_{+}\phi_2 +d_{0}\gamma\Gamma^{\gamma-1}_{+}Z_0=0 \quad \mbox{in}\quad\R^2.
\end{equation}
Using these assumptions, one can realise that the Inner-Outer system in \eqref{inprobexp} and \eqref{poisreduced} admits a solution of the form $(\phi_{in},\phi_{out})=(\phi_1+\phi_2,\phi_{out})$, provided we further establish the coupled conditions 
\[d_1\left[\bar{B}[\phi_2]+H(\phi_1+\phi_2,\phi_{out})\right]=0, \quad d_2\left[\bar{B}[\phi_2]+H(\phi_1+\phi_2,\phi_{out})\right]=0.\]
We defer the analysis of these conditions to Section \ref{reduced} and begin by examining the equation for $\phi_2$ in \eqref{problemforphi2}. Turning to polar coordinates $y=re^{i\theta}, r=|y|,$ it is easy to see that a smooth solution to this problem is given by 
\begin{equation}\label{explicitphi2}
\phi_2(r)=-d_0 Z_0 (r) \int_{\xi_0}^{r}\frac{\dd s}{sZ_0(s)^2} \int_{0}^{s} \gamma\Gamma^{\gamma-1}_{+}(\rho) Z_{0}^2(\rho)\rho \, \dd \rho\eqqcolon d_0\hat{\phi_2}(r),
\end{equation}
where $\xi_0$ is the unique root of $Z_0$ in $(0,1).$

An important remark is that since $\operatorname{supp} \Gamma^{\gamma-1}_{+} \subset B_1(0)$, for any $s>1$ we get \[\int_{0}^{s}\gamma\Gamma^{\gamma-1}_{+}(\rho)Z_0(\rho)^2\rho \, \dd \rho = \int_{0}^{1}\gamma\Gamma^{\gamma-1}_{+}(\rho)Z_0(\rho)^2\rho \, \dd \rho. \]
With this in mind, for any $r>0$ we deduce that
\begin{equation}\label{estimateforphi2}
|\phi_2(r)| \leq C|d_0|\log(2+r).
\end{equation}
 The following Proposition is the main result of this section.
\begin{prop}\label{propsec7}
Let $\sigma > 0$ and $\beta \in (0,1)$. For $\ve>0$ sufficiently small, there exist constants $0<\sigma^{*}<\sigma$,  functions $\phi_{out}(x)$, $\phi_1(y)$, $\phi_2(y)$ solving \eqref{poisreduced}, \eqref{problemforphi1} and \eqref{problemforphi2} respectively, and scalars $d_j, \, j=1,2$, such that
 \begin{equation}\label{boundsforsoln}
    \left\|(1+|x|^2)^{-1}\phi_{out}\right\|_{\infty}\leq C\ve^{1+\sigma^{*}}, \quad \|\phi_1\|_{*,\sigma,B_{\rho}}<C\ve, \quad |d_0|<C \ve^{1+\sigma}|\log\ve|^{2+2\sigma},
 \end{equation}
 where $y=\frac{z}{\ve},$ \,$\rho=\frac{2\delta}{\ve|\log\ve|^2}$ and $\|\cdot\|_{*,\sigma,B_{\rho}}$ is defined in \eqref{starnorm}.
\end{prop}
\begin{proof}
The problem can be formulated as finding a pair $(\phi_1,d_0)$ solving the fixed point problem 
\begin{equation}\label{fixedfinal}
(\phi_1,d_0) =\tilde{\mathcal{F}}(\phi_1,d_0),
\end{equation}
where 
\begin{equation*}
\begin{aligned}
\phi_1&=T[H(\phi_1+\phi_2,\phi_{out})]\in \mathcal{X}_{*}\defeq\{\phi_{in} \in C^{2,\beta}(\R^2)\ : 
  \|\phi_{in}\|_{*, \sigma,B_{\rho}} <+\infty\},\\
d_0&=\gamma_0\int_{R^2} \left(H(\phi_1+\phi_2,\phi_{out})+\bar{B}[\phi_2]\right)Z_0, \quad \gamma_0^{-1}=\int_{\R^2}\gamma\Gamma^{\gamma-1}_{+}Z_0^2.
\end{aligned}
\end{equation*}
In the above, $T$ is the linear operator built in Proposition \ref{prop1}, \,$\phi_2(y)=d_0\hat{\phi}_2(y)$ is given in \eqref{explicitphi2}, while $\phi_{out}(x)$ was obtained in Section \ref{subsecouter}, satisfying \eqref{boundphiout} and \eqref{phioutP}.

By means of the Banach's fixed point Theorem, we want to find a solution in the ball 
\begin{equation}\label{contractionball}
\mathcal{B}=\{(\phi_1,d_0)\in \mathcal{X}_{*}\times \R: \phi_1(y_1,y_2)=\phi_1(y_1,-y_2), \,\, \|\phi_1\|_{*,\sigma,B_\rho}\leq \hat{C}\ve, \,\, |d_0|\leq \hat{C}\ve^{1+\sigma}|\log\ve|^{2+2\sigma}\},
\end{equation}
for some constant $\hat{C}>0$ to be chosen.

To carry out the fixed point argument, we firstly estimate
\begin{equation*}
\begin{aligned}
\left|H(\phi_1+\phi_2,\phi_{out})\right| &\leq C\Gamma^{\gamma-2}_{+}\left(|\tilde{\eta}_{\delta}\phi_1|^2+|\tilde{\eta}_{\delta}\phi_2|^2+|\phi_{out}|^2\right) + \ve^2 S(\Psi_\alpha)+\left(\gamma\Gamma^{\gamma-1}_{+}+b_0(y)\right)|\phi_{out}|\\
&\leq C\ve |y|\Gamma^{\gamma}_{+},
\end{aligned}
\end{equation*}
thus we immediately obtain \[\|H(\phi_1+\phi_2,\phi_{out})\|_{2+\sigma,\beta,B_{\rho}} \leq C \ve.
\]
In addition, using \eqref{estimateforphi2} one finds
\begin{equation*}
\begin{aligned}
\left|\bar{B}[\phi_2]\right| &\leq C\ve\left(|y||D^{2}_{y}\phi_2|+|D_y \phi_2|\right)+\mathcal{O}\left(\ve |y|\Gamma^{\gamma-2}_{+}\right)\phi_2\\
&\leq \frac{C\ve|d_0|}{1+|y|}+\mathcal{O}\left(\ve|y|\Gamma^{\gamma-2}_{+}\right)|d_0|\log(2+|y|),\\
\end{aligned}
\end{equation*}
so we further have
\[\left\|\bar{B}[\phi_2]\right\|_{2+\sigma,\beta,B_{\rho}}\leq C\ve.\]
Invoking Proposition \ref{prop1}, we readily get
\begin{equation}\label{firsterr71}
\big\|\phi_1\left[H(\phi_1+\phi_2,\phi_{out})+\bar{B}[\phi_2]\right]\big\|_{*,\sigma,B_{\rho}}\leq\big\| H(\phi_1+\phi_2,\phi_{out})+\bar{B}[\phi_2]\big\|_{2+\sigma,\beta,B_{\rho}}\leq C_1\ve,
\end{equation}
for some $C_1>0$.

Next, we compute 
\begin{equation}\label{estimateBphi2Zo}
\begin{aligned}
\int_{B_{\rho}}\bar{B}[\phi_2]Z_0 
&\leq C\ve  |d_0|\int_{B_{\rho}}\frac{\log(2+|y|)}{1+|y|}+ \int_{B_{\rho}}\mathcal{O}\left(\ve|y|\Gamma^{\gamma-2}_{+}\right)|d_0|\log(2+|y|) Z_0\\
&\leq \mathcal{O}\left(\frac{|d_0|}{|\log\ve|}\right),
\end{aligned}
\end{equation}
and 
\begin{equation*}
\begin{aligned}
\int_{B_{\rho}}H\left(\phi_1+\phi_2,\phi_{out}\right)Z_0 &\leq C\int_{B_{\rho}}\Gamma^{\gamma-2}_{+}\left(|\tilde{\eta}_{\delta}(\phi_1+\phi_2)|^2+|\phi_{out}|^2\right)Z_0 +C\int_{B_{\rho}} \ve^2 S(\Psi_{\alpha})Z_0\\
&+C\int_{B_{\rho}}\left(\gamma\Gamma^{\gamma-1}_{+}+\mathcal{O}\left(\ve|y|\Gamma^{\gamma-2}_{+}\right)\right)|\phi_{out}|Z_0.
\end{aligned}
\end{equation*}
Since $Z_0$ is a radial function, using Proposition \ref{errorprop1} we observe that \[\int_{B_{\rho}}\ve^2 S(\Psi_{\alpha})Z_0 =\mathcal{O}\left(\ve^2|\log\ve|\right), \qquad \int_{B_{\rho}}\Gamma^{\gamma-2}_{+}\left(|\tilde{\eta}_{\delta}(\phi_1+\phi_2)|^2+|\phi_{out}|^2\right)Z_0=\mathcal{O}\left(\ve^2\right).  \]
However, due to  $\int_{B_{\rho}} \gamma\Gamma^{\gamma-1}_{+} Z_0 = \mathcal{O}(1)$, the bound for $\phi_{out}$ in \eqref{boundphiout} gives
\[\int_{B_{\rho}} \left(\gamma\Gamma^{\gamma-1}_{+}+b_0(y)\right)Z_0|\phi_{out}|=\mathcal{O}(\ve^{1+\sigma^{*}}),\]
for some $0<\sigma^{*}<\sigma$.

This estimate reveals a strong coupling between the Inner and Outer problems, since for $0<\sigma^{*}<\sigma$ we have that $\ve^{1+\sigma^{*}} \gg \ve^{1+\sigma}|\log\ve|^{1+2\sigma}$ for all small $\ve>0$. This is precisely where the reduced form of the Outer problem \eqref{poisreduced} provided by Lemma \ref{lemmaout} becomes vital, as to overcome this issue we use \eqref{phioutP} together with the expansion \[\phi_{out}(x)=\phi_{out}(P)+\mathcal{O}\left(\ve |y|\|\phi_{out}\|_{\infty}\right) \quad \mbox{as} \quad  \ve\to 0,\] which is valid in the region $\left\{x\in\R^2:|x-P|<\frac{1}{|\log\ve|^2}\right\},$ to eventually deduce that
\begin{equation}\label{estimateHZ0afterP}
\int_{B_{\rho}} H\left(\phi_1+\phi_2,\phi_{out}\right)Z_0 \leq C\ve^2|\log\ve|.
\end{equation}
Combining \eqref{estimateBphi2Zo} and \eqref{estimateHZ0afterP}, we write
\[\int_{B_{\rho}}\left(H(\phi_1+\phi_2,\phi_{out}+\bar{B}[\phi_2]\right)Z_0=\mathcal{O}\left(\frac{|d_0|}{|\log\ve|}\right),\]
while employing Proposition \ref{prop1} again, one obtains 
\begin{equation}\label{secerr71}
\begin{aligned}
 |d_0|&=\gamma_{0}\int_{\R^2} \left(H\left(\phi_1+\phi_2,\phi_{out}\right)+\bar{B}[\phi_2]\right)Z_0\\
 &\leq \mathcal{O}\left(\frac{|d_0|}{|\log\ve|}\right)+\mathcal{O}\left(\ve^{\sigma}|\log\ve|^{1+2\sigma}\right)\left\|H(\phi_1+\phi_2,\phi_{out})+\bar{B}[\phi_2]\right\|_{2+\sigma,\beta,B_{\rho}}\\
 &\leq C_2 \ve^{1+\sigma}|\log\ve|^{2+2\sigma},
\end{aligned}
\end{equation}
for some $C_2>0.$

Collecting \eqref{firsterr71} and \eqref{secerr71}, it is clear that $\tilde{\mathcal{F}}(\mathcal{B})\subset \mathcal{B}$, as long as we choose the constant $\hat{C}>0$ in \eqref{contractionball} sufficiently large, independent of $\ve>0$.  

To complete the proof, it remains to establish that $\tilde{\mathcal{F}}$ in \eqref{fixedfinal} is a contraction mapping in $\mathcal{B}$. For this purpose, we consider the pairs $(\phi^{1}_{1},d^{1}_0), (\phi^{2}_{1},d^{2}_{0}) \in \mathcal{B},$ together with the corresponding $\phi^{1}_{2}, \phi^{2}_{2}$ in the decompositions of $\phi_{in}^{1}$ and $\phi_{in}^{2}$. 

We recall that for $j=1,2$ it holds
\[d^{j}_{0}=\gamma_{0}\int_{R^2} \left(H(\phi^{j}_{1}+\phi^{j}_{2},\phi_{out}[\phi^{j}_{1}+\phi^{j}_{2}])+\bar{B}[\phi^{j}_{2}]\right)Z_0, \quad \gamma_0^{-1}=\int_{\R^2}\gamma\Gamma^{\gamma-1}_{+}Z_0^2.\]
Defining $\phi_{out}^j=\phi_{out}[\phi_1^j+\phi_2^j]$ and $H^{j}=H\left(\phi_1^j +\phi_2^j,\phi_{out}[\phi_1^j+\phi_2^j]\right)$ for $j=1,2,$ we use \eqref{outeroperatorphiin} to derive the estimate
\begin{equation}\label{phioutcontraction}
\|\phi^1_{out}-\phi^2_{out}\|_{\infty}\leq C\left(\ve^{\sigma}|\log\ve|^{\mu}\|\phi^1_{1}-\phi^2_{1}\|_{*,\sigma,B_{\rho}}+|\log\ve|^{\mu}|d_0^1-d_0^2|\right),
\end{equation}
for some $\mu>0.$

Furthermore, since 
\begin{equation*}
      H^1-H^2=\mathcal{N}\left(\tilde{\eta}_{\delta}(\phi^{1}_{1}+\phi^{1}_{2})+\phi_{out}^{1}\right)-\mathcal{N}\left(\tilde{\eta}_{\delta}(\phi^{2}_{1}+\phi^{2}_{2})+\phi_{out}^{2}\right)+ \left(\gamma\Gamma^{\gamma-1}_{+}+b_0(y)\right)\left(\phi_{out}^{1}-\phi_{out}^{2}\right),
\end{equation*}
we further obtain
\begin{equation*}
\begin{aligned}
|H^1-H^2|&<C\Gamma^{\gamma-2}_{+}\left(\tilde{\eta}_{\delta}\left(|\phi_{1}^1-\phi_{1}^2|^2+|\phi_2^1-\phi_2^2|^2\right)+|\phi^1_{out}-\phi^2_{out}|^2\right)\\
&+\left(\gamma\Gamma^{\gamma-1}_{+}
+\mathcal{O}\big(\ve|y|\Gamma^{\gamma-2}_{+}\big)\right)|\phi^1_{out}-\phi^2_{out}|,
\end{aligned}
\end{equation*}
which allows to conclude that
\begin{equation}\label{Esth1minush2}
\begin{aligned}
\|H^1-H^2\|_{2+\sigma,\beta,B_{\rho}}&\leq C \left(\|\phi_{out}^1-\phi_{out}^2\|_{\infty}+\|\phi^{1}_{1}-\phi^{2}_{1}\|^2_{*,\sigma,B_{\rho}}+|d^{1}_{0}-d^{2}_{0}|^2+\|\phi_{out}^1-\phi_{out}^2\|^2_{\infty}\right)\\
&+C\ve\left(\|\phi^{1}_{1}-\phi^{2}_{1}\|_{*,\sigma,B_{\rho}}+|d^{1}_{0}-d^{2}_{0}|+\|\phi_{out}^1-\phi_{out}^2\|_{\infty}\right).
\end{aligned}
\end{equation}
Moreover, due to \eqref{explicitphi2} and \eqref{estimateforphi2} we have \[\bar{B}[\phi^{1}_{2}-\phi^{2}_{2}]=\mathcal{O}(|d^{1}_{0}-d^{2}_{0}|)\bar{B}[\hat{\phi}_2], \quad \mbox{where} \quad \bar{B}[\hat{\phi}_{2}]\leq \frac{C\ve}{2+|y|}+\mathcal{O}\left(\ve|y|\Gamma^{\gamma-2}_{+}\right)\hat{\phi}_2,\]
hence we get
\begin{equation}\label{estforBfinal}
\left\|\bar{B}[\phi^{1}_{2}-\phi^{2}_{2}]\right\|_{2+\sigma,\beta,B_{\rho}} \leq \frac{\mathcal{O}\left(|d_0^1-d_0^2|\right)}{\ve^{\sigma}|\log\ve|^{2+2\sigma}}.
\end{equation}
On the other hand, it holds
\begin{equation}\label{finalestimates1}
\begin{aligned}
    |d^{1}_{0}-d^{2}_{0}|&\leq \mathcal{O}\left(\frac{\log(2+\rho)}{1+\rho^{\sigma}}\right)\left\|H^1-H^2+\bar{B}[\phi^{1}_{2}-\phi^{2}_{2}]\right\|_{2+\sigma,\beta,B_{\rho}}+\mathcal{O}\left(\frac{|d_0^1-d_0^2|}{|\log\ve|}\right)\\
    &\leq C\ve^{\sigma}|\log\ve|^{1+2\sigma}\Bigg[\|\phi_{out}^1-\phi_{out}^2\|_{\infty}+\|\phi^{1}_{1}-\phi^{2}_{1}\|^2_{*,\sigma,B_{\rho}}+|d^{1}_{0}-d^{2}_{0}|^2+\|\phi_{out}^1-\phi_{out}^2\|^2_{\infty}\\
&+\ve\Big(\|\phi^{1}_{1}-\phi^{2}_{1}\|_{*,\sigma,B_{\rho}}+|d^{1}_{0}-d^{2}_{0}|+\|\phi_{out}^1-\phi_{out}^2\|_{\infty}\Big)\Bigg]+\mathcal{O}\left(\frac{|d_0^1-d_0^2|}{|\log\ve|}\right),
\end{aligned}
\end{equation}
thus making use of Proposition \ref{prop1}, \eqref{Esth1minush2}, \eqref{estforBfinal} and \eqref{finalestimates1}, we find 
\begin{equation}\label{finalestimates2}
\begin{aligned}
&\big\|T\left[H^1-H^2+\bar{B}[\phi_2^1-\phi_2^2]\right]\big\|_{*,\sigma,B_{\rho}}\leq C\big\|H^1-H^2+\bar{B}[\phi_2^1-\phi_2^2]\big\|_{2+\sigma,\beta,B_{\rho}}\\
&\leq C\Bigg[\|\phi_{out}^1-\phi_{out}^2\|^2_{\infty}+\|\phi^{1}_{1}-\phi^{2}_{1}\|^2_{*,\sigma,B_{\rho}}+|d^{1}_{0}-d^{2}_{0}|^2 +\|\phi_{out}^1-\phi^2_{out}\|_{\infty}\\&\hspace{10mm}+\ve\left(\|\phi^{1}_{1}-\phi^{2}_{1}\|_{*,\sigma,B_{\rho}}+|d^{1}_{0}-d^{2}_{0}|+\|\phi_{out}^1-\phi^2_{out}\|_{\infty}\right)\\
&\hspace{10mm}+\frac{1}{|\log\ve|}\left(\|\phi^{1}_{1}-\phi^{2}_{1}\|_{*,\sigma,B_{\rho}}+|d_0^1-d_0^2|+\|\phi_{out}^1-\phi_{out}^2\|_{\infty}\right)\Bigg].
\end{aligned}
\end{equation}
To this end, combining \eqref{phioutcontraction}, \eqref{finalestimates1} and \eqref{finalestimates2} one can verify that $\tilde{\mathcal{F}}$ is a contraction mapping in $\mathcal{B}$ for all sufficiently small $\ve>0$, which yields the existence of the desired fixed point.
\end{proof}

\section{Solving the Reduced Problem}\label{reduced}
In Section \ref{Section projected}, we have established the existence of a solution $(\phi_{in},\phi_{out})=(\phi_1+\phi_2,\phi_{out})$ to the coupled system of equations
\begin{equation*}
\begin{cases}
\Delta_{y}\phi_{in}+\gamma\Gamma^{\gamma-1}_{+}\phi_{in}+\bar{B}[\phi_{in}]+H(\phi_{1}+\phi_{2},\phi_{out})=\sum\limits_{j=1}^{2}d_j\gamma\Gamma^{\gamma-1}_{+}Z_j \quad \mbox{in} \quad B_{\rho},\\
L_{x}[\phi_{out}]+L_{x}[\tilde{\eta}_{\delta}\phi_{in}]-\tilde{\eta}_{\delta}L_{x}[\phi_{in}]=0\quad \text{in }  \R^2.
\end{cases}
\end{equation*}
For the reader's convenience, we recall that  $\rho = \frac{2\delta}{\ve|\log\ve|^2}$, $\bar{B}$ is given in \eqref{defofBinner} and $Z_j, j=1,2$ are defined in \eqref{elemkernel}, while the solution $(\phi_{in},\phi_{out})$ and the relevant estimates are contained in Proposition \ref{propsec7} and \eqref{boundsforsoln}.

To conclude the proof of Theorem \ref{maintheorem}, it remains to obtain an actual solution to the coupled Inner-Outer system in \eqref{inprobexp} and \eqref{poisreduced}, by further requiring that 
\begin{equation}\label{zerofunctionals}
 d_1\left[H(\phi_1+\phi_2,\phi_{out})+\bar{B}[\phi_2]\right]=0, \quad d_2\left[H(\phi_1+\phi_2,\phi_{out})+\bar{B}[\phi_2]\right]=0,
 \end{equation}
where 
\begin{equation*}
    \begin{aligned}
        &d_j=\gamma_j \int_{\R^2}\left(H(\phi_1+\phi_2,\phi_{out})+\bar B[\phi_2]\right)Z_j, \quad \gamma_{j}^{-1}=\int_{\R^2}\gamma\Gamma^{\gamma-1}_{+}Z_j^2, \quad j=1,2,\\
&H(\phi_1+\phi_2,\phi_{out})=\mathcal{N}\left(\tilde{\eta}_{\delta}\phi_{in}+\phi_{out}\right)+\ve^2S(\Psi_\alpha)+\left(\gamma\Gamma^{\gamma-1}_{+}+\mathcal{O}\big(\ve|y|\Gamma^{\gamma-2}_{+}\big)\right)\phi_{out}.
    \end{aligned}
\end{equation*}
To attain these conditions, we claim that the rotational speed $\alpha$ in \eqref{rotansatz} must be chosen as in \eqref{rotspeed}. 

To see this, using Proposition \ref{prop1} we initially observe that for $j=1,2,$ it holds  
\begin{equation*}
\begin{aligned}
d_j\left[H(\phi_1+\phi_2,\phi_{out})+\bar{B}[\phi_2]\right]&=\gamma_{j}\int_{B_{\rho}} \left(H(\phi_1+\phi_2,\phi_{out})+\bar{B}[\phi_2]\right)Z_{j}\\
&\hspace{4mm}+\mathcal{O}\left(\ve^{1+\sigma}|\log\ve|^{2+2\sigma}\right)\big\|H(\phi_1+\phi_2,\phi_{out})+\bar{B}[\phi_2]\big\|_{2+\sigma,\beta,B_{\rho}}\\
&=\gamma_{j}\int_{B_{\rho}} \left(H(\phi_1+\phi_2,\phi_{out})+\bar{B}[\phi_2]\right)Z_{j} +\mathcal{O}\left(\ve^{1+\sigma}\right),
\end{aligned}
\end{equation*}
for $\sigma>0$ as in \eqref{boundsforsoln}.

As a result, by continuity we get that the conditions in \eqref{zerofunctionals} reduce to 
\[\gamma_{j}\int_{B_{\rho}} \left(H(\phi_1+\phi_2,\phi_{out})+\bar{B}[\phi_2]\right)Z_j =\mathcal{O}\left(\ve^{1+\sigma}\right), \quad j=1,2.\]
Moreover, due to the estimate
\[\gamma_{j}\int_{B_{\rho}} \left(\mathcal{N}\left(\tilde{\eta}_{\delta}\phi_{in}+\phi_{out}\right)+\left(\gamma\Gamma^{\gamma-1}_{+}+\mathcal{O}\big(\ve|y|\Gamma^{\gamma-2}_{+}\big)\right)\phi_{out}\right)Z_j =\mathcal{O}(\ve^{1+\sigma}),\]
we realise that it remains to ensure that \[\int_{B_{\rho}} \ve^2 S(\Psi_\alpha) Z_j=\mathcal{O}(\ve^{1+\sigma}), \quad j=1,2.\]
In light of Proposition \ref{errorprop1}, the error of approximation satisfies $\ve^2S(\Psi_{\alpha})(\ve y_1,\ve y_2)=\ve^2S(\Psi_{\alpha})(\ve y_1,-\ve y_2)$ and has the expansion 
\begin{equation}\label{errorfinalsection}
\begin{aligned}
    \ve^2 S(\Psi_\alpha)&= \frac{3r_0h^2+r_0^2}{2h(h^2+r_0^2)^{\frac 3 2}}\ve y_1\Gamma^{\gamma}_{+} +\gamma\Gamma^{\gamma-1}_{+}\Bigg[\ve y_1|\log\ve|\left(-c_1\nu'(1)-\frac{\alpha r_0 h}{\sqrt{h^2+r_0^2}}\right)\\
    &+\ve y_1 \left(c_1\Gamma+\frac{h}{\sqrt{h^2+r_0^2}}\partial_{y_1}H_{2\ve}(P)\right) +\mathcal{O}\left(\ve^2|\log\ve||y|^2\right) \Bigg] +\mathcal{O}\left(\ve^2|\log\ve|^2|y|^2\Gamma^{\gamma-2}_{+}\right).
    \end{aligned}
\end{equation}
Since $Z_2=\frac{\partial\Gamma}{\partial{y_2}}$ is odd in $y_2,$ it is then immediate due to symmetry that \[\int_{B_{\rho}} \ve^2 S(\Psi_{\alpha}) Z_2 =0.\]
Upon testing \eqref{errorfinalsection} with $Z_1$, to satisfy the condition $\int_{B_{\rho}} \ve^2 S(\Psi_{\alpha})Z_1 =\mathcal{O}(\ve^{1+\sigma})$ we choose 
\begin{equation*}
\alpha= \frac{-\nu'(1)}{2(h^2+r_0^2)} +\tilde\alpha,
\end{equation*}
where $\tilde\alpha$ takes the form 
\[\tilde\alpha=\frac{1}{|\log\ve|}\left[\frac{\int_{B_{\rho}}-\left[\frac{3r_0h^2+r_0^2}{2h(h^2+r_0^2)^{\frac 3 2}}\Gamma_{+}^{\gamma}+\gamma\Gamma^{\gamma-1}_{+} \left(c_1\Gamma +\frac{h}{\sqrt{h^2+r_0^2}}\partial_{y_1}H_{2\ve}(P)\right)\right]y_1Z_1}{\int_{B_{\rho}}\gamma\Gamma^{\gamma-1}_{+}y_1Z_1}\right](1+o(1)),\]
with $o(1) \to 0$ as $\ve \to 0.$
This concludes the proof of Theorem \ref{maintheorem}.
\section{Proof of Theorem \ref{theorem 2}}\label{multihelices}
In this section, we provide some ideas concerning the proof of Theorem \ref{theorem 2}. We recall that the equation we need to solve reads 
\begin{equation}\label{eqfinalsec}
-\nabla_{x}\cdot(K\nabla_x\Psi)=f\left(\Psi-\frac{\alpha}{2}|\log\ve||x|^2\right) \quad \mbox{in} \,\, \R^2, \quad f(s)=\frac{1}{\ve^2}\left(s+\nu'(1)|\log\ve|\right)^{\gamma}_{+},
\end{equation}
where $\gamma>3,\,s_{+}=\max(0,s),$ and $K(x)$ is defined in \eqref{operL}.

In polar coordinates $x=re^{i\theta}, r=|x|,$ and considering the standard basis $e_r=(\cos\theta,\sin\theta), \, e_{\theta}=(-\sin\theta,\cos\theta)$, we have that $K e_r=\frac{h^2}{h^2+r^2}e_{r}$ and $K e_{\theta}=e_{\theta},$ thus equation \eqref{eqfinalsec} can be equivalently written as 
\begin{equation}\label{finaleqpolar}
-\frac{h^2}{r}\partial_{r}\left(\frac{r}{h^2+r^2}\partial_r \Psi\right)-\frac{1}{r^2}\partial_{\theta\theta}\Psi=f\left(\Psi-\frac{\alpha}{2}|\log\ve|r^2\right) \quad \mbox{in} \,\, (r,\theta)\in(0,\infty)\times[0,2\pi).
\end{equation}
For any integer $k\geq 2$, it is then immediate that equation \eqref{finaleqpolar} is invariant under the rotation $\theta\mapsto \theta + \frac{2\pi}{k}$ and even symmetry $\theta \mapsto -\theta,$ due to the spherical symmetry of \eqref{eqfinalsec}.

To obtain the solution predicted by Theorem \ref{theorem 2}, we set  $z=|z|e^{i\theta}$ and work in the function space of dihedral symmetry
\begin{equation*}
\hat{\mathcal{X}}\defeq \left\{\Psi \in L^{\infty}_{\text{loc}}\left(\R^2\right): \Psi\left(ze^{\frac{i2\pi}{k}}\right)=\Psi(z), \,\, \Psi(z)=\Psi(\bar{z}), \,\, z=(z_1,z_2)\in\R^2\right\}.
\end{equation*}
We then introduce the profile \[\tilde{H}_{1\ve}(z)=\hat{\Gamma}\left(\frac{z}{\ve}\right)\left(1+c_1z_1+c_2|z|^2\right)+\frac{r_0^3}{2h(h^2+r_0^2)^{\frac 3 2}}H_{1\ve}(z),\]
with $\hat{\Gamma}\left(\frac{z}{\ve}\right), \, c_1, \, c_2$ and $H_{1\ve}(z)$ given in Proposition \ref{propapprox}, and introduce the approximate solution 
\[\Psi_{\alpha}(x)=\frac{\alpha}{2}|\log\ve|r_0^2 +\sum_{j=1}^{k}\eta_{\delta}\left(x-r_0e^{\frac{i2\pi(j-1)}{k}}\right)\tilde{H}_{1\ve}\left(ze^{\frac{i2\pi(j-1)}{k}}\right) +H_{2\ve}(x) \in \hat{\mathcal{X}},\]
where the cut-off function $\eta_{\delta}$ is defined in \eqref{definitioneta}.

The remainder of the proof of Theorem \ref{theorem 2} is identical to that of Theorem \ref{maintheorem}, hence we omit the details.

\bibliographystyle{plain}
\bibliography{references}
\end{document}